\documentclass[11pt]{amsart}
\usepackage[utf8]{inputenc}
\usepackage{mathptmx} 
\usepackage{caption}
\usepackage[scaled=0.90]{helvet} 
\usepackage{courier} 
\usepackage[T1]{fontenc}
\usepackage{float}

\usepackage[
hypertexnames=false, colorlinks, 
linkcolor=black,citecolor=black,urlcolor=black, linktocpage=true
]%
{hyperref} 

\usepackage{amsmath}
\usepackage{wrapfig}
\usepackage{amsthm}
\usepackage{amsfonts}
\usepackage{amssymb}
\usepackage{graphicx}

\usepackage{mathrsfs}

\usepackage{color}
\definecolor{paata}{rgb}{1,0,0} 

\usepackage{bbm}
\DeclareMathOperator*{\esssup}{ess\,sup}
\DeclareMathOperator*{\essinf}{ess\,inf}
\newtheorem{theorem}{Theorem}
\newtheorem*{theoremV}{Theorem A}
\newtheorem*{theoremB}{Theorem B}

\newtheorem{lemma}{Lemma}
\newtheorem{corollary}{Corollary}

\newtheorem{remark}{Remark}

\newtheorem{proposition}{Proposition}

\usepackage[margin=0.96in]{geometry}

\begin{document}

\title{Boundary value problem and the Ehrhard inequality}

\author{Paata Ivanisvili}
\address{Department of Mathematics,  Kent State University, Kent, OH 44240, USA
}
\email{ivanishvili.paata@gmail.com}


\makeatletter
\@namedef{subjclassname@2010}{
  \textup{2010} Mathematics Subject Classification}
\makeatother




%
%


\begin{abstract}
let $I, J\subset \mathbb{R}$ be closed intervals, and let $H$ be $C^{3}$ smooth real valued function on  $I\times J$ with nonvanishing $H_{x}$ and $H_{y}$.  Take any fixed positive numbers  $a,b$, and let $d\mu$ be a probability measure with finite moments and absolutely continuous with respect to Lebesgue measure.  We show that for the inequality
$$
\int_{\mathbb{R}^{n}} \esssup_{y \in \mathbb{R}^{n}}\;  H\left( f\left(\frac{x-y}{a}\right),g\left(\frac{y}{b}\right)\right)d\mu (x) \geq H\left(\int_{\mathbb{R}^{n}}fd\mu, \int_{\mathbb{R}^{n}}gd\mu \right)
$$
to hold  for all Borel functions $f,g$ with values in $I$ and $J$ correspondingly it is necessary that
$$
a^{2}\frac{H_{xx}}{H_{x}^{2}}+(1-a^{2}-b^{2})\frac{H_{xy}}{H_{x}H_{y}}+b^{2}\frac{H_{yy}}{H_{y}^{2}}\geq 0,
$$
$|a-b|\leq 1$, $a+b\geq 1$ and $\int_{\mathbb{R}^{n}}xd\mu=0$ if $a+b>1$.  Moreover, if $d\mu$ is a gaussian measure then the necessary condition becomes sufficient.  This extends Pr\'ekopa--Leindler and  Ehrhard inequalities  to an arbitrary function $H(x,y)$.  As an immediate application we obtain the new proof of the Ehrhard inequality.   In particular, we  show that in the class of even probability measures with smooth positive density and finite moments the Gaussian measure is the only one which satisfies the functional form of the Ehrhard inequality on the real line with their own distribution functions.   
\end{abstract}

\date{}
\maketitle

\section{Introduction}

Let $I, J \subset \mathbb{R}$ be closed intervals. Set $\Omega:= I\times J$,  and let $H : \Omega \to \mathbb{R}$. Fix some $n\geq1$.
Let $d\mu$ be a probability measure on $\mathbb{R}^{n}$ and absolutely continuous with respect to the Lebesgue measure. For simplicity we will always assume that $H \in C^{3}(\Omega)$ and $\int_{\mathbb{R}^{n}}\|x\|^{5} d\mu<\infty$. 

  In this paper we address the following question: what is the necessary and sufficient condition on $H$,  positive real numbers $a, b$ and a measure $d\mu$  such that the following inequality holds
\begin{align}\label{ragaca}
\int_{\mathbb{R}^{n}} \esssup_{y \in \mathbb{R}^{n}}\;  H\left( f\left(\frac{x-y}{a}\right),g\left(\frac{y}{b}\right)\right)d\mu (x) \geq H\left(\int_{\mathbb{R}^{n}} f d\mu, \int_{\mathbb{R}^{n}} g d\mu \right) 
\end{align}
for all Borel measurable $f,g$ with values in $I$ and $J$ correspondingly.   Essential supremum in (\ref{ragaca}) is taken with respect to the Lebesgue measure. 
Our main result is the following theorem.  
\begin{theorem}\label{mth1}
Suppose that $H_{x}$ and $H_{y}$ never vanish in $\Omega$. For  inequality (\ref{ragaca}) to hold it is necessary that
\begin{align}\label{numbers} 
  a^{2} \frac{H_{xx}}{H_{x}^{2}} + (1-a^{2}-b^{2})\frac{H_{xy}}{H_{x} H_{y}} + b^{2} \frac{H_{yy}}{H_{y}^{2}} \geq 0, 
\end{align}
$|1-a^{2}-b^{2}|  \leq 2ab$,  and $\int_{\mathbb{R}^{n}} x d\mu =0$  if  $a+b>1$. 
Moreover, if $d\mu(x)$ is a Gaussian measure  then the above conditions are also sufficient.
\end{theorem}
 By Gaussian measure we mean a probability measure of the form 
 \begin{align}\label{gauss-measure}
 \exp(-xAx^{T}+bx^{T} +c)\,dx \quad \text{for some}\; n \times n\; \text{matrix} \quad  A>0, \; b \in \mathbb{R}^{n} \quad \text{and } \quad  c \in \mathbb{R}. 
 \end{align} 
 The symbols $H_{x}, H_{y}, H_{xx}, H_{xy}$ and $H_{yy}$ denote partial derivatives.  $x^{T}$ denotes transpose of the row vector $x \in \mathbb{R}^{n}$. Constraint $\left|1-a^{2}-b^{2}\right| \leq 2ab$ on number $a,b>0$  can be rewritten as $a+b \geq 1$ and $|a-b|\leq 1$. Moreover, if  $a+b>1$ then it is necessary that $\int_{\mathbb{R}^{n}} xd\mu$  is the  zero vector. 
  

In the applications usually $a+b=1$. Therefore the most important condition  the reader needs to keep in mind is the partial differential inequality (PDI) in (\ref{numbers}). We should also notice an {\em independence from the dimension}, i.e., the necessity conditions follow from the one dimensional case $n=1$ of (\ref{ragaca}), and for the Gaussian measures (\ref{numbers}) is sufficient for (\ref{ragaca})  to hold  for all $n \geq 1$. 

Partial differential inequality (\ref{numbers}) first time appeared in the PhD thesis of the author (see Theorem~3.0.22 in \cite{PhD}), and later in \cite{IV} (see Corollary 5.2 in \cite{IV}) as a sufficient condition for inequality (\ref{ragaca}) to hold in case of the Gaussian measure  with supremum in (\ref{ragaca}) and smooth compactly supported functions $f,g$. Namely, it was proved in \cite{IV}  that if $H_{x}, H_{y}$ are nonvanishing, and $H$ satisfies (\ref{numbers}), then the following inequality holds 
 \begin{align}\label{weaker}
 \int_{\mathbb{R}^{n}} \sup_{ax+by=t}\;  H\left( f(x),g(y)\right)d\mu(t) \geq H\left(\int_{\mathbb{R}^{n}} f d\mu, \int_{\mathbb{R}^{n}} g d\mu \right) 
 \end{align} 
 for all smooth compactly supported functions $f,g$ with values in $I, J$ correspondingly and the Gaussian measure $d\mu$. In this case we need the assumption that $I, J$ contain the origin. 
 
 In the present  paper we obtain a certain extension of (\ref{weaker}) by using different techniques. The  first  immediate extension is that we have (\ref{ragaca}) with essential supremum and Borel measurable functions\footnote{Since the proof of (\ref{weaker}) in \cite{IV} essentially uses intermediate value theorem for continuous functions $f,g$ to verify property (3.3) in \cite{IV},  it is unclear how to extend the argument of \cite{IV} to discontinuous functions.}. Our second extension is that we obtain if and only if characterization, moreover we obtain the necessity part for {\em almost arbitrary} probability measures $d\mu$. Our approach to (\ref{ragaca})  sheds light to a question about {\em optimizers}, and it  provides us with  some quantitative version of (\ref{ragaca}) (see Lemma~\ref{lem-5} and Lemma~\ref{lem-7}), and,  more importantly, it shows a hidden  link between  two different type of PDEs considered in \cite{IV} (see PDE (1.3) and (1.5) in \cite{IV}). 
  
 Our argument, at some point, uses a remarkable Theorem~A obtained, for example, in \cite{neeman1, ledoux1, IV} (we also present the sketch of the proof of Theorem~A in the Appendix). The proof of Theorem~A  relies on the classical maximum principle for parabolic PDEs unlike the proof of (\ref{weaker}) in \cite{IV} which uses a  subtle maximum principle used first time by Borell~\cite{bor1} ({\em hill property} in ~\cite{IV}, and Lemma~1 in \cite{bar1}), and it does not follow at all from the classical maximum principle. Hence, in particular, we obtain the new proof of the Ehrhard inequality from the classical maximum principle. We should also mention that authors in \cite{PN} ask whether one can deduce the Ehrhard inequality solely from Theorem~A. The current paper gives an affirmative answer. 

  In Section~\ref{proof1} we present the proof of Theorem~\ref{mth1}. In Section~\ref{app}, using arguments from exterior differential systems, we will linearize PDE, the left hand side of (\ref{numbers}), and we will explain how to find functions $H$ for which inequality in  (\ref{numbers}) is equality.  Besides, we will illustrate various applications of the theorem. 

 \subsection*{Acknowledgements} I am  grateful to Christos Saroglou who initiated this project and with whom I had many discussions. He should be considered as co-author (despite his insistence to the contrary).  I am extremely thankful to the Kent State Analysis Group especially  Fedor Nazarov who gave me some valuable suggestions in obtaining the necessity part, and Artem Zvavitch for providing C.~Borell's lecture notes. The talk given by Grigoris Paouris on the Informal Analysis Seminar at Kent State University  served as a guide and  inspiration for the present article.

\section{Proof of Theorem~\ref{mth1}}\label{proof1}
\subsection{The necessity condition}\label{nec-1}
First we notice that if (\ref{ragaca}) holds for some $n \geq 1$ then it holds for $n=1$. Indeed we can test (\ref{ragaca}) on the functions $f(x_{1},x_{2}, \ldots.x_{n}) = \tilde{f}(x_{1})$, $g(x_{1},x_{2},\ldots, x_{n})=\tilde{g}(x_{1})$ for some Borel functions  $\tilde{f},\tilde{g}$ from  $\mathbb{R}$ to $I, J$ correspondingly. In what follows we will assume that $n=1$.  
Finiteness of the fifth moment together with the  Lebesgue dominated convergence theorem implies that 
\begin{align}\label{tends-to-zero}
R^{5}\int_{R}^{\infty}d\mu \to 0 \quad \text{and} \quad  R^{5}\int_{-\infty}^{-R}d\mu \to 0 \quad \text{as} \quad  R \to \infty. 
\end{align}

We need several technical lemmas. We fix a number $\alpha \in (0,1/3)$  close to $\frac{1}{3}$  which will be determined later. 
\begin{lemma}\label{decay}
\begin{align*}
\left| \int^{\pm \infty}_{\pm \varepsilon^{-\alpha}}|t|d\mu  \right| = o(\varepsilon^{4\alpha}) \quad \text{and} \quad \left| \int^{\pm \infty}_{\pm \varepsilon^{-\alpha}}t^{2}d\mu  \right| = o(\varepsilon^{3\alpha}),
\end{align*}
as $\varepsilon \to 0$. 
\end{lemma}
\begin{proof}
We have
\begin{align*}
\left| \int^{\pm \infty}_{\pm \varepsilon^{-\alpha}}|t|d\mu  \right|\leq \varepsilon^{4\alpha} \left| \int^{\pm \infty}_{\pm \varepsilon^{-\alpha}}|t|^{5}d\mu \right| = o(\varepsilon^{4\alpha}).
\end{align*}
Similarly for the second integral. 
\end{proof}

Let $(u,v)$ be the point in the interior of $\Omega$.  Let  $H_{u}=H_{u}(u,v)$ and $H_{v}=H_{v}(u,v)$.
\begin{lemma}\label{first-form}
If $H$ satisfies (\ref{ragaca}) then 
\begin{align}\label{main-form}
\frac{p\frac{H_{vv}}{H_{v}^{2}}+pq + q \frac{H_{uu}}{H_{u}^{2}}+\frac{H_{uu}H_{vv}-H_{uv}^{2}}{H_{u}^{2}H_{v}^{2}}}{ \frac{H_{uu}}{H_{u}^{2}} +p-2\frac{H_{uv}}{H_{u}H_{v}}
 +\frac{H_{vv}}{H_{v}^{2}}+q } \geq pa^{2}+qb^{2}
\end{align}
for all real numbers $p$ and $q$ such that  $p+q+\frac{H_{uu}}{H_{u}^{2}}-2\frac{H_{uv}}{H_{u} H_{v}} + \frac{H_{vv}}{H_{v}^{2}} < 0$. 
\end{lemma}
\begin{proof}

Let $\delta>1$ be a number  which will be determined later.  

 We consider the following test functions $(f,g)$: 
\begin{align}
&f(x) = u + \varepsilon \frac{ \varphi_{\varepsilon,\delta}(a\, x)}{H_{u}}+\varepsilon^{2} \frac{p\, \varphi^{2}_{\varepsilon,\delta}(a\, x)}{2H_{u}};\label{test-1}\\
&g(y) = u + \varepsilon \frac{ \varphi_{\varepsilon,\delta}(b\, y)}{H_{v}}+\varepsilon^{2} \frac{q\, \varphi^{2}_{\varepsilon,\delta}(b\, y)}{2H_{v}};\label{test-2}\\
&\text{where} \quad \varphi_{\varepsilon,\delta}(t) = 
\begin{cases}
-\delta \varepsilon^{-\alpha}&  t\leq -\delta \varepsilon^{-\alpha}\;;\\
t & -\delta \varepsilon^{-\alpha} \leq t \leq \varepsilon^{-\alpha}\;;\\
\varepsilon^{-\alpha} &   \varepsilon^{-\alpha} \leq t\,.
\end{cases} \label{test-3}
\end{align} 
We notice that $|\varphi_{\varepsilon,\delta}(t)| \leq \delta \varepsilon^{-\alpha}$. 
Since $\alpha <1$ it is clear that  $f:\mathbb{R}\to I$ and $g : \mathbb{R} \to J$ for all $0\leq  \varepsilon \leq \varepsilon_{0}$  where $\varepsilon_{0}$ is a small number. Ideally we want to choose $\varphi_{\varepsilon, \delta}(t) =t$ for all $t \in \mathbb{R}$ but then the image of $(f,g)$ will escape from the rectangle $\Omega$. 

Let $\int_{\mathbb{R}} t d\mu = \tau$ and $\int_{\mathbb{R}} t^{2} d\mu(t)=\beta$. Choose $\alpha \in (0,\frac{1}{3})$ so that $4\alpha>1$.  Notice that for each fixed $\delta>1$  by  (\ref{tends-to-zero}) and Lemma~\ref{decay} we have 
\begin{align*}
\int_{\mathbb{R}}f d\mu = u+\varepsilon \frac{a\tau }{H_{u}}+\varepsilon^{2}\frac{ p a^{2} }{2H_{u}} \beta  + o(\varepsilon^{2}) \quad \text{and} \quad \int_{\mathbb{R}}g d\mu = v+\varepsilon \frac{b\tau }{H_{v}}+\varepsilon^{2}\frac{ q b^{2} }{2H_{v}} \beta  + o(\varepsilon^{2}) \quad \text{as} \quad \varepsilon \to 0.
\end{align*}
Using the fact that $H \in C^{3}(I\times J)$ by Taylor's formula we obtain 
\begin{align*}
&H(s,t)=H(u,v)+(s-u)H_{u}+(t-v)H_{v}+\frac{1}{2}\left((s-u)^{2}H_{uu}+2(s-u)(t-v)H_{uv}+(t-v)^{2}H_{vv} \right)\\
&+O((|s-u|+|t-v|)^{3}).
\end{align*}
Taking $s=\int f d\mu$ and $t = \int g d\mu$ we obtain
\begin{align}
&H\left(\int_{\mathbb{R}} f d\mu, \int_{\mathbb{R}} g d\mu \right)=H(u,v)+\varepsilon \tau (a+b)+\frac{\varepsilon^{2}}{2} \beta  (p a^{2} +qb^{2})\label{latala}\\
&+\frac{1}{2}\varepsilon^{2} \tau^{2}\left(a^{2}\frac{H_{uu}}{H_{u}^{2}}+2ab\frac{H_{uv}}{H_{u}H_{v}}+b^{2}\frac{H_{vv}}{H_{v}^{2}} \right)+
o(\varepsilon^{2}) \quad \text{as} \quad \varepsilon \to 0. \nonumber
\end{align}
On the other taking $s=f(x)$ and $t=g(y)$ we obtain  
\begin{align*}
&H(f(x),g(y))=H(u,v)+\varepsilon(\varphi_{\varepsilon,\delta}(a\,x)+\varphi_{\varepsilon,\delta}(b\,y))+\\
&\frac{\varepsilon^{2}}{2}  \left(p\varphi^{2}_{\varepsilon,\delta}(a\,x)+ q\varphi^{2}_{\varepsilon,\delta}(b\,y)+ \frac{H_{uu}}{H_{u}^{2}} \varphi^{2}_{\varepsilon,\delta}(a\,x)+2\frac{H_{uv}}{H_{u} H_{v}}\varphi_{\varepsilon,\delta}(a\,x)\varphi_{\varepsilon,\delta}(b\,y)
 +\frac{H_{vv}}{H_{v}^{2}}\varphi^{2}_{\varepsilon,\delta}(b\,y)\right)+O(\varepsilon^{3(1-\alpha)}).
\end{align*}
Since $\alpha <1/3$ we have  $O(\varepsilon^{3(1-\alpha)}) = o(\varepsilon^{2})$. First we should compare small order terms in order to get a restriction on $\tau$. 

Since $f,g$ are continuous clearly the essential supremum of the integrand in (\ref{ragaca}) becomes $\sup_{ax+by=z} H(f(x),g(y))$. 
Thus, introducing new variables $\tilde{x} = a\,x, \tilde{y}=b\,y$ and using the fact that supremum of the sum is at most the sum of the supremums,  we obtain that  (\ref{ragaca}) implies   the inequality 
\begin{align}\label{miss1}
\varepsilon \int_{\mathbb{R}}\sup_{x+y=t} (\varphi_{\varepsilon,\delta}(x)+\varphi_{\varepsilon,\delta}(y)) d\mu(t) \geq \varepsilon \tau(a+b) + o(\varepsilon).
\end{align}

Notice that $\sup_{x+y=t} (\varphi_{\varepsilon,\delta}(x)+\varphi_{\varepsilon,\delta}(y)) = t$ for $t \in [(1-\delta)\varepsilon^{-\alpha}, 2\varepsilon^{-\alpha}]$,  and it is bounded as $C \varepsilon^{-\alpha}$ otherwise. Therefore (\ref{miss1}) implies that $\tau \geq (a+b)\tau$. On the other hand we can considered the new test functions $\tilde{\varphi}_{\varepsilon,\delta}=-\varphi_{\varepsilon,\delta}$, and  we can  obtain the opposite inequality $-\tau \geq -(a+b)\tau$. This implies that if $a+b>1$ then $\tau=0$. Notice that in the case $a+b=1$, without loss of generality, we can assume that $\int_{\mathbb{R}} t d\mu=0$. Indeed, we can test inequality (\ref{weaker}) on the translated functions $f_{c}(x)=f(x-c) $ and $g_{c}(y)=g(y-c)$. After change of variables in (\ref{weaker}), and using $ac+bc=c$, we obtain that (\ref{weaker}) holds with initial test functions $f,g$ and shifted measure $\mu_{c}(\cdot ) = \mu(\cdot+c)$. Clearly we can choose $c \in \mathbb{R}$ so that $\int_{\mathbb{R}}t d\mu_{c}(t)=0$.

In what follows we assume $\tau=0$, and therefore,  the terms involving $\tau$ in  (\ref{latala}) are zero. Inequality  (\ref{ragaca}) implies that 
\begin{align}
&\varepsilon \int_{\mathbb{R}}\sup_{x+y=t} (\varphi_{\varepsilon,\delta}(x)+\varphi_{\varepsilon,\delta}(y)) d\mu(t) +\nonumber\\
&\frac{\varepsilon^{2}}{2} \int_{\mathbb{R}}\sup_{x+y=t}\left[ \left( \frac{H_{uu}}{H_{u}^{2}} +p\right)\varphi^{2}_{\varepsilon,\delta}(x)+2\frac{H_{uv}}{H_{u}H_{v}}\varphi_{\varepsilon,\delta}(x)\varphi_{\varepsilon,\delta}(y)
 +\left(\frac{H_{vv}}{H_{v}^{2}}+q \right)\varphi^{2}_{\varepsilon,\delta}(y)\right]d\mu(t)  \geq  \nonumber \\
 &\frac{\varepsilon^{2}}{2} \beta (pa^{2} +  qb^{2}) + o(\varepsilon^{2}). \label{final-inequality}
\end{align}

 Since $\int_{\mathbb{R}} t d\mu=0$, (\ref{tends-to-zero}) and Lemma~\ref{decay}  we obtain  
\begin{align*}
\int_{\mathbb{R}}\sup_{x+y=t} (\varphi_{\varepsilon,\delta}(x)+\varphi_{\varepsilon,\delta}(y)) d\mu(t) = o(\varepsilon) \quad \text{as} \quad \varepsilon \to 0.
\end{align*}
Set 
\begin{align*}
\psi_{\varepsilon,\delta}(t) \overset{\mathrm{def}}{=}\sup_{x+y=t}\left[ \left( \frac{H_{uu}}{H_{u}^{2}} +p\right)\varphi^{2}_{\varepsilon,\delta}(x)+2\frac{H_{uv}}{H_{u}H_{v}}\varphi_{\varepsilon,\delta}(x)\varphi_{\varepsilon,\delta}(y)
 +\left(\frac{H_{vv}}{H_{v}^{2}}+q \right)\varphi^{2}_{\varepsilon,\delta}(y)\right].
\end{align*}
We remind that $p$ and $q$ are chosen in such a way that $p+\frac{H_{uu}}{H_{u}^{2}}+q+ \frac{H_{vv}}{H_{v}^{2}}-2\frac{H_{uv}}{H_{u} H_{v}}  < 0$. 

We need the following lemma. 
\begin{lemma}\label{second-form}
Let $\delta>1$ be such that   for all $s$,  $\frac{1}{\delta} \leq s \leq \delta$ we have 
\begin{align}\label{about-delta}
 \left( p+\frac{H_{uu}}{H_{u}^{2}}\right)s^{2}-2\frac{H_{uv}}{H_{u} H_{v}}s+q+ \frac{H_{vv}}{H_{v}^{2}} <0,
\end{align}
then there exist sufficiently small positive constants $c$ and $\varepsilon_{0}>0$ such that 
\begin{align}\label{convergence}
\psi_{\varepsilon,\delta}(t)  =\frac{p\frac{H_{vv}}{H_{v}^{2}}+pq + q \frac{H_{uu}}{H_{u}^{2}}+\frac{H_{uu}H_{vv}-H_{uv}^{2}}{H_{u}^{2}H_{v}^{2}}}{ \frac{H_{uu}}{H_{u}^{2}} +p-2\frac{H_{uv}}{H_{u}H_{v}}
 +\frac{H_{vv}}{H_{v}^{2}}+q } \cdot t^{2},
\end{align}
for all real $|t| \leq c \varepsilon^{-\alpha}$ and all $\varepsilon \leq \varepsilon_{0}$.   
\end{lemma}
Before we proceed to the proof of the lemma, let us mention that Lemma~\ref{first-form}  follows from Lemma~\ref{second-form}. Indeed, first we choose $\delta>1$ such that (\ref{about-delta}) holds. Such choice is possible because of the continuity and the assumption on the numbers $p$ and $q$. Lemma~\ref{second-form}, (\ref{final-inequality}), (\ref{tends-to-zero}), Lemma~\ref{decay} and the fact that $|\psi_{\varepsilon,\delta}(t)| \leq C\varepsilon^{-2\alpha}$ on the complement of the interval $[-c\varepsilon^{-\alpha}, c\varepsilon^{-\alpha}]$ imply (\ref{main-form}). 
Thus it remains to prove Lemma~\ref{second-form}.
\begin{proof}
Set 
\begin{align*}
w(x,y)= \left( \frac{H_{uu}}{H_{u}^{2}} +p\right)\varphi^{2}_{\varepsilon,\delta}(x)+2\frac{H_{uv}}{H_{u}H_{v}}\varphi_{\varepsilon,\delta}(x)\varphi_{\varepsilon,\delta}(y)
 +\left(\frac{H_{vv}}{H_{v}^{2}}+q \right)\varphi^{2}_{\varepsilon,\delta}(y).
\end{align*}
 We should describe behavior of $w(x,y)$ on the red line $x+y=t$ (see Figure~\ref{fig:dom}).   If 
$2\varepsilon^{-\alpha}\geq t \geq (1-\delta) \varepsilon^{-\alpha}$  then the line $x+y=t$ will cross the sides $DA$  and $DC$ of the rectangle $ABCD$ as it is shown on Figure~\ref{fig:dom}.

\begin{figure}
\includegraphics[scale=1]{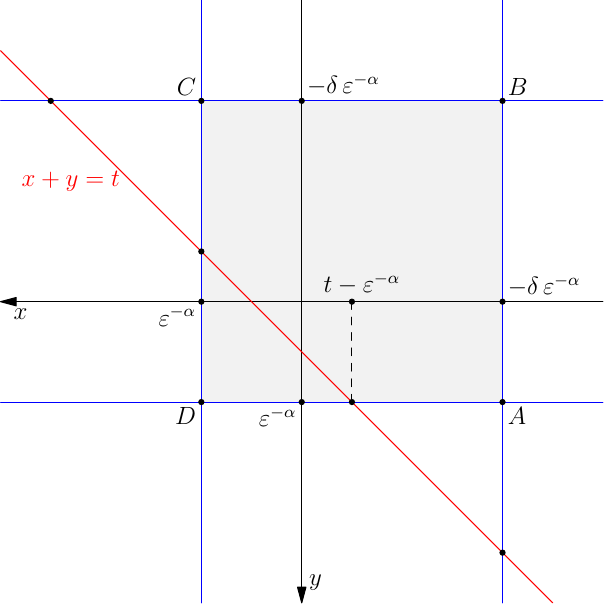}
\caption{Domain of the function $w(x,y)$}
\label{fig:dom}
\end{figure}
We have 
\begin{align*}
w(x,t-x) = 
\begin{cases}
\varepsilon^{-2\alpha} \left[ \left( \frac{H_{uu}}{H_{u}^{2}} +p\right)\delta^{2}-2\frac{H_{uv}}{H_{u}H_{v}}\delta
 +\left(\frac{H_{vv}}{H_{v}^{2}}+q \right)\right] & x \leq -\delta \varepsilon^{-\alpha};\\
  \left( \frac{H_{uu}}{H_{u}^{2}} +p\right)x^{2}+2\frac{H_{uv}}{H_{u}H_{v}}x \varepsilon^{-\alpha}
 +\left(\frac{H_{vv}}{H_{v}^{2}}+q \right) \varepsilon^{-2\alpha}  & t - \varepsilon^{-\alpha} \geq x  \geq -\delta \varepsilon^{-\alpha};\\
  \left( \frac{H_{uu}}{H_{u}^{2}} +p\right)x^{2}+2\frac{H_{uv}}{H_{u}H_{v}}x (t-x)
 +\left(\frac{H_{vv}}{H_{v}^{2}}+q \right)(t-x)^{2}  & t - \varepsilon^{-\alpha} \leq x  \leq \varepsilon^{-\alpha};\\
  \left( \frac{H_{uu}}{H_{u}^{2}} +p\right)\varepsilon^{-2\alpha}+2\frac{H_{uv}}{H_{u}H_{v}} \varepsilon^{-\alpha}(t-x)
 +\left(\frac{H_{vv}}{H_{v}^{2}}+q \right) (t-x)^{2}  & \varepsilon^{-\alpha} \leq x \leq t+\delta \varepsilon^{-\alpha};\\
 \varepsilon^{-2\alpha} \left[ \left( \frac{H_{uu}}{H_{u}^{2}} +p\right)-2\frac{H_{uv}}{H_{u}H_{v}}\delta
 +\left(\frac{H_{vv}}{H_{v}^{2}}+q \right) \delta^{2}\right] & x \geq t+\delta \varepsilon^{-\alpha}.
 \end{cases}
\end{align*}
Notice that because of the assumption (\ref{about-delta}) the values of $w(x,t-x)$ approach to negative infinity as $-c \varepsilon^{-2\alpha}$ if $x \leq -\delta \varepsilon^{-\alpha}$ or $x \geq t+\delta \varepsilon^{-\alpha}$ where $c>0$ is some constant.  

If $t - \varepsilon^{-\alpha} \geq x \geq -\delta \varepsilon^{-\alpha}$ then let us reparametrize the function $w(x,t-x)$ as follows $x= -\varepsilon^{-\alpha}s$ where $1-t \varepsilon^{\alpha} \leq s \leq \delta$.  Then 
\begin{align*}
w(-\varepsilon^{-\alpha}s, t+\varepsilon^{-\alpha}s)= \varepsilon^{-2\alpha} \left[ \left( \frac{H_{uu}}{H_{u}^{2}} +p\right)s^{2}-2\frac{H_{uv}}{H_{u}H_{v}}s
 +\left(\frac{H_{vv}}{H_{v}^{2}}+q \right)\right].
\end{align*}

Clearly if $t \leq (1-\frac{1}{\delta})\varepsilon^{-\alpha}$  then $1-t \varepsilon^{\alpha} \geq \frac{1}{\delta}$ and  by (\ref{about-delta}) the maximal value of $w(x,t-x)$ behaves as $-c\varepsilon^{-2\alpha}$ for some $c>0$ on the interval $t-\varepsilon^{-\alpha} \geq x \geq -\delta\, \varepsilon^{-\alpha}$. Behavior of $w(x,t-x)$ on the interval $\varepsilon^{-\alpha} \leq x \leq t+ \delta  \varepsilon^{-\alpha}$ is completely symmetric to the previous case. 

Since the value of the map $x \mapsto w(x,t-x)$ goes to negative infinity at the endpoints of the interval $[t-\varepsilon^{-\alpha}, \varepsilon^{-\alpha}]$ as $\varepsilon \to 0$, one can check that the maximum of the function $w(x,t-x)$ is attained inside of the interval $[t-\varepsilon^{-\alpha}, \varepsilon^{-\alpha}]$  at the   point
\begin{align*}
x_{0} = -\frac{\left( \frac{H_{uv}}{H_{u}H_{v}}-\frac{H_{vv}}{H_{v}^{2}}-q\right)t}{\frac{H_{uu}}{H_{u}^{2}}+p-2\frac{H_{uv}}{H_{u}H_{v}}+\frac{H_{vv}}{H_{v}^{2}}+q}.
\end{align*}

Thus,  we have 
\begin{align}\label{foru-supremum}
\sup_{x \in \mathbb{R}} w(x,t-x) = \max_{x \in [t-\varepsilon^{-\alpha}, \varepsilon^{-\alpha}]} w(x,t-x)=w(x_{0}, t-x_{0})=
\frac{p\frac{H_{vv}}{H_{v}^{2}}+pq + q \frac{H_{uu}}{H_{u}^{2}}+\frac{H_{uu}H_{vv}-H_{uv}^{2}}{H_{u}^{2}H_{v}^{2}}}{ \frac{H_{uu}}{H_{u}^{2}} +p-2\frac{H_{uv}}{H_{u}H_{v}}
 +\frac{H_{vv}}{H_{v}^{2}}+q } \cdot t^{2},
\end{align}
for all $|t| \leq c \varepsilon^{-\alpha}$ and all $\varepsilon \leq \varepsilon_{0}$  where $\varepsilon_{0}$ and $c$ are  some sufficiently small numbers. Therefore,  we obtain (\ref{convergence}).

\end{proof}

\end{proof}

\begin{lemma}
Inequality (\ref{main-form}) holds for all real $p$ and  $q$  with $p+q+\frac{H_{uu}}{H_{u}^{2}}-2\frac{H_{uv}}{H_{u}H_{v}}+\frac{H_{vv}}{H_{v}^{2}} <0$ if and only if  
\begin{align*}
 \left(\frac{1-a^{2}-b^{2}}{2ab}\right)^{2} \leq 1 \quad \text{and} \quad a^{2}\frac{H_{uu}}{H_{u}^{2}} + (1-a^{2}-b^{2})\frac{H_{uv}}{H_{u}H_{v}}+b^{2}\frac{H_{vv}}{H_{v}^{2}} \geq 0. 
\end{align*}
\end{lemma}
\begin{proof}
Let us rewrite (\ref{main-form}) as follows 
\begin{align*}
&M(p,q)=(p,q) \;C\; (p,q)^{T}+p \left[ a^{2} \left( \frac{H_{uu}}{H_{u}^{2}}-2\frac{H_{uv}}{H_{u} H_{v}}\right)+(a^{2}-1)\frac{H_{vv}}{H_{v}^{2}}\right]+\\
&q\left[b^{2}\left( \frac{H_{vv}}{H_{v}^{2}}-2\frac{H_{uv}}{H_{u} H_{v}}\right)+(b^{2}-1)\frac{H_{uu}}{H_{u}^{2}} \right]-\frac{H_{uu}H_{vv}-H_{uv}^{2}}{H_{u}^{2}H_{v}^{2}}\geq 0
\end{align*}
on the half plane 
\begin{align}\label{halfp}
p+q+\frac{H_{uu}}{H_{u}^{2}}-2\frac{H_{uv}}{H_{u}H_{v}}+\frac{H_{vv}}{H_{v}^{2}} <0,
\end{align}
 where 
\begin{align*}
C = \begin{pmatrix}
a^{2} & \frac{a^{2}+b^{2}-1}{2}\\
\frac{a^{2}+b^{2}-1}{2} & b^{2}
\end{pmatrix}.
\end{align*}
In order for the quadric form $M(p,q)$ to be nonnegative on the half plane it is necessary that $C\geq 0$. Indeed, suppose there is $(p_{0}, q_{0}) \neq (0,0)$ such that $(p_{0}, q_{0})C(p_{0}, q_{0})^{T}<0$. Without loss of generality we can assume that $p_{0}+q_{0}<0$, otherwise if $p_{0}+q_{0}\geq 0$ we can consider a new pair $(\tilde{p}, \tilde{q})=(-p_{0}, -q_{0})$ and perturb it slightly, if necessary, to ensure that $\tilde{p}+\tilde{q}<0$. Finally taking $(p_{\lambda}, q_{\lambda}):=\lambda (p_{0}, q_{0})$ we can choose $\lambda>0$ sufficiently large so that (\ref{halfp}) holds. On the other hand $\lim_{\lambda \to \infty}\frac{M(p_{\lambda}, q_{\lambda})}{\lambda^{2}}= (p_{0}, q_{0})C (p_{0}, q_{0})^{T}<0$. Thus we must have $C\geq 0$, and the latter condition, namely, $\mathrm{det}\,  C\geq 0$ gives the constraint  on the numbers $a,b>0$. 

 Notice that $M(p,q)$ is nonnegative on the boundary of the half plane, i.e.,
\begin{align*}
M\left(-q-\frac{H_{uu}}{H_{u}^{2}}+2\frac{H_{uv}}{H_{u}H_{v}}-\frac{H_{vv}}{H_{v}^{2}}, q\right) = \frac{(H_{vv}H_{u}-H_{uv}H_{v}+H_{u}H_{v}^{2}q)^{2}}{H_{u}^{2}H_{v}^{4}}\geq 0. 
\end{align*} 

  Next we consider the case when  $\left( \frac{1-a^{2}-b^{2}}{2ab}\right)^{2}<1$.
  
Let $(p_{0},q_{0})$ be the vertex of the paraboloid $M$, i.e., $\nabla M (p_{0},q_{0})=0$. The direct computations show that 
\begin{align*}
M(p_{0},q_{0})=\frac{\left( a^{2}\frac{H_{uu}}{H_{u}^{2}}+(1-a^{2}-b^{2})\frac{H_{uv}}{H_{u}H_{v}}+b^{2}\frac{H_{vv}}{H_{v}^{2}}\right)^{2}}{4a^{2}b^{2}\left( \left(\frac{1-a^{2}-b^{2}}{2ab}\right)^{2}-1\right)},
\end{align*}
and 
\begin{align}\label{movkvdi}
p_{0}+q_{0}+\frac{H_{uu}}{H_{u}^{2}}-2\frac{H_{uv}}{H_{u}H_{v}}+\frac{H_{vv}}{H_{v}^{2}} = \frac{a^{2}\frac{H_{uu}}{H_{u}^{2}}+(1-a^{2}-b^{2})\frac{H_{uv}}{H_{u}H_{v}}+b^{2}\frac{H_{vv}}{H_{v}^{2}}}{2a^{2}b^{2} \left( 1- \left(\frac{1-a^{2}-b^{2}}{2ab}\right)^{2}\right) }.
\end{align}
Therefore $M \geq 0$ in the halfplane (\ref{halfp}) if and only if the right hand side of (\ref{movkvdi}) is nonnegative

If $\left(\frac{1-a^{2}-b^{2}}{2ab}\right)^{2}=1$ then  $\det (\mathrm{Hess})\; M=0$, therefore, $M$ is the developable surface, i.e., $M$ is linear along some straight line segments. The direction $(x_{0},y_{0})$ of these straight line segments satisfy the equation $\mathrm{Hess}\; M \, (x_{0},y_{0})^{T}=2C\, (x_{0},y_{0})^{T}=(0,0)$, i.e., $(x_{0},y_{0})=\left( \frac{a^{2}+b^{2}-1}{2}, -a^{2}\right)$ which, clearly, is not parallel to the boundary of the halfplane (\ref{halfp}). Condition $\left(\frac{1-a^{2}-b^{2}}{2ab} \right)^{2}=1$ implies that $b=|1-a|$ or $b=a+1$. If $b=|1-a|$ then $(x_{0}, y_{0})\cdot (1,1)=\frac{a^{2}+b^{2}-1}{2}-a^{2} = \frac{b^{2}-a^{2}-1}{2}<0$.  In this case
\begin{align*}
\lim_{\lambda \to \infty} \frac{M(p+\lambda x_{0},q+ \lambda y_{0})}{\lambda} =a \left(a^{2}\frac{H_{uu}}{H_{u}^{2}}+(1-a^{2}-b^{2})\frac{H_{uv}}{H_{u}H_{v}}+b^{2}\frac{H_{vv}}{H_{v}^{2}} \right),
\end{align*}
and the latter expression must be nonnegative. Finally, if $b=a+1$ then $(x_{0}, y_{0})\cdot (1,1) = \frac{b^{2}-a^{2}-1}{2}>0$. In this case 
\begin{align*}
\lim_{\lambda \to \infty} \frac{M(p-\lambda x_{0},q- \lambda y_{0})}{\lambda}=a \left(a^{2}\frac{H_{uu}}{H_{u}^{2}}+(1-a^{2}-b^{2})\frac{H_{uv}}{H_{u}H_{v}}+b^{2}\frac{H_{vv}}{H_{v}^{2}} \right),
\end{align*}
and the latter expression must be nonnegative. Since $M$ is nonnegative on the boundary of the halfplane (\ref{halfp}), this finishes the proof of the lemma 
\end{proof}

\subsection{The sufficiency for the Gaussian measure}\label{suf-1}
Our main ingredient will be  a subtle Theorem~3 from \cite{IV}.  Let us  precisely formulate it in  the way we will use it.  
Let $k, k_{1}, k_{2}$ and $k_{3}$ be some positive integers with $k\geq k_{j}$, $j=1,2,3$. Let $A_{j}$ be $k\times k_{j}$ size matrices of full rank for $j=1,2$ and $3$. Set $A=(A_{1},A_{2},A_{3})$ to be $k \times (k_{1}+k_{2}+k_{3})$ size. Let $B : \Omega \subset \mathbb{R}^{3} \to \mathbb{R}$ be in $C^{2}(\Omega)$ where $\Omega$ is a closed bounded  rectangular domain, i.e., $\Omega = I_{1} \times I_{2} \times I_{3}$ where $I_{1}, I_{2}, I_{3}$ are closed subintervals in $\mathbb{R}$.  Let $C$ be a positive definite $k \times k$ matrix. Set 
\begin{align*}
d\gamma_{C}(x) = \frac{1}{\sqrt{(2 \pi)^{k}\det(C)}} e^{-\frac{|C^{-1/2}x|^{2}}{2}}dx. 
\end{align*}

 By $A^{*}$ we denote the transpose of the matrix $A$. Let $x \in \mathbb{R}^{k}$ be a row vector, i.e., $x=(x_{1},\ldots, x_{k})$.  Let $A^{*}CA \bullet \mathrm{Hess}\, B$ denotes $(\sum k_{j}) \times (\sum k_{j})$ matrix  $\{A_{i}^{*}CA_{j} \partial_{ij}B\}_{i,j=1}^{3}$, i.e., $A^{*}CA\bullet \mathrm{Hess}\, B$ is constructed by the blocks  $A_{i}^{*}CA_{j}\partial_{ij}B$. 
\begin{theoremV}\label{ivo-1}
$A^{*}CA \bullet \mathrm{Hess}\; B \geq 0$ on $\Omega$ if and only if 
\begin{align}
&\int_{\mathbb{R}^{k}} B(u_{1}(xA_{1}), u_{2}(xA_{2}), u_{3}(xA_{3}) )d\gamma_{C}(x) \geq  \label{ivo-2}\\
&B\left( \int_{\mathbb{R}^{k_{1}}} u_{1}(y\sqrt{A_{1}^{*}CA_{1}}) d\gamma_{k_{1}}(y),  \int_{\mathbb{R}^{k_{2}}} u_{2}(y\sqrt{A_{2}^{*}CA_{2}}) d\gamma_{k_{2}}(y), \int_{\mathbb{R}^{k_{3}}} u_{3}(y\sqrt{A_{3}^{*}CA_{3}}) d\gamma_{k_{3}}(y)\right)\nonumber
\end{align}
for all Borel measurable $u_{j} :\mathbb{R}^{k_{j}} \to I_{j}$, $j=1,2,3$. 
\end{theoremV}
Here $d\gamma_{n}(y)=\frac{e^{-|x|^{2}/2}}{(2\pi)^{n/2}}dy$ is the standard Gaussian measure on $\mathbb{R}^{n}$. Sometimes we will omit dependence on dimension $n$ and we will write $d\gamma$, and the corresponding dimension of the Gaussian measure will be clear from the context. 

The theorem was formulated for smooth $B$ and compactly supported functions $u_{j}$, $j=1,2,3$. We should mention that the theorem still remains true for $B \in C^{2}(\Omega)$ and for smooth bounded $u_{j}$. The proof proceeds absolutely in the same way as in \cite{IV}. For the conveninence of the reader we decided to sketch the proof in Section~\ref{atrakebs} (see Appendix). 

 In order to obtain (\ref{ivo-2}) for all  Borel measurable functions  $u_{j} : \mathbb{R}^{k_{j}} \to I_{j}$  we approximate pointwise almost everywhere by smooth bounded functions $u^{n}_{j}$ such that $\mathrm{Im}(u^{n}_{j}) \in I_{j}$. Finally, the Lebesgue  dominated convergence theorem justifies the result (notice that all functions $u_{1}, u_{2}, u_{3}$ and $B$ are uniformly bounded).

We should also mention that inequality (\ref{ivo-2}) for the  function $B(x,y,z)=x^{p}y^{q}z^{r}$ recovers the reverse  Young's inequality for convolutions with sharp constants, and the latter  was used in \cite{brascamp-lieb1} in obtaining the Pr\'ekopa--Leindler inequality. In our case the situation is slightly different.  We will be using (\ref{ivo-2}) for some sequence of functions $B^{R}$, matrices $C^{R}$, $A^{R}$, test functions $(u^{R}_{1}, u^{R}_{2}, u^{R}_{3})$,  and very special sequence of  Gaussian measures $d\gamma^{R}$ where $R\geq 1$. Finally, in the limit $R \to \infty$ we will obtain (\ref{ragaca}). 

Further in obtaining the sufficiency condition, without loss of generality we can assume that  $d\mu = d\gamma$. The case of the arbitrary Gaussian measure (\ref{gauss-measure}) follows by testing (\ref{ragaca}) on the shifts and dilates of $f,g$ and change of variables in (\ref{ragaca}).  Next we consider two different cases, when $|1-a^{2}-b^{2}|=2ab$ and when $|1-a^{2}-b^{2}|<2ab$. To the first case we refer as {\em parabolic case} and the second case we call  {\em elliptic case}. These names originate from studying the solutions of the partial differential inequality in (\ref{numbers}), see Remark~\ref{hypokrat}.

\subsubsection{Parabolic case.} 
In this subsection we consider the case when 
\begin{align}\label{heat}
\left( \frac{1-a^{2}-b^{2}}{2ab}\right)^{2}=1.
\end{align}
(\ref{heat}) holds if and only if $a=|1-b|$ or $a=1+b$.  
 Without loss of generality we will assume that $H>\delta$ for some $\delta>0$. Moreover, we  can choose $\delta>0$ so that  $|H_{x}|, |H_{y}| >\delta >0$. 
  Set 
\begin{align}\label{measure}
d\gamma_{p,q,x}(y) = \left( \frac{p}{2\pi}\right)^{n/2} e^{-\frac{|py+qx|^{2}}{2p}}dy  \quad \text{for} \quad x, y \in \mathbb{R}^{n}. 
\end{align}
The choice of the numbers $p,q$ will be specified later. So far we  assume that $p>0$. Notice that $d\gamma_{p,q,x}(y)$ is a probability measure on $\mathbb{R}^{n}$. We need the following lemma. 
\begin{lemma}\label{lem-5}
For any $1>\alpha>\beta>0$ with $\alpha+\beta>1$ there exists $R_{0} = R_{0}(\alpha,\beta, H)$ such that for all $R>R_{0}$ we have 
\begin{align}
&\int_{\mathbb{R}^{n}}\left( \int_{\mathbb{R}^{n}} H^{R}\left(f\left(\frac{x-y}{a}\right), g\left(\frac{y}{b}\right) \right)d\gamma_{p,q,x}(y)\right)^{\frac{1}{R-R^{\alpha}}}d\gamma_{n}(x) \geq  \nonumber  \\
&H^{\frac{R}{R-R^{\alpha}}}\left(\int_{\mathbb{R}^{n}} f \left( x \sqrt{1+\frac{ 1 }{a^{2} R^{\beta}}}\right)d\gamma(x),\int_{\mathbb{R}^{n}}g\left(x \sqrt{1+\frac{ 
1  }{b^{2} R^{\beta}}} \right)d\gamma(x)\right),  \label{mq1}
\end{align}
where $p=R^{\beta}$, $q=-b R^{\beta}$ if $a=|1-b|$ and $q=bR^{\beta}$ if $a=b+1$. 
\end{lemma}
Before we proceed to the proof of the lemma let us explain that the lemma implies the desired result (\ref{ragaca}).  It is clear that if $R \to \infty$ then the right hand side of (\ref{mq1}) tends to $H(\int  f d\gamma, \int g\gamma)$.  We claim that the left hand side of (\ref{mq1}) tends to $\int_{\mathbb{R}^{n}} \esssup_{y} H (f((x-y)/a), g(y/a)) d\gamma(x)$. Indeed, let 
\begin{align*}
\varphi_{R}(x) = \left( \int_{\mathbb{R}^{n}} H^{R}\left(f\left(\frac{x-y}{a}\right), g\left(\frac{y}{b}\right) \right)d\gamma_{p,q,x}(y)\right)^{\frac{1}{R-R^{\alpha}}}. 
\end{align*}
We claim that $\varphi_{R}(x) \to \esssup_{y} H (f((x-y)/a), g(y/a))$ a.e. as $R \to \infty$. Notice that 
\begin{align*}
\varphi_{R}(x) \leq  \left( \esssup_{y} \; H(f((x-y)/a), g(y/a)) \right)^{{\frac{R}{R-R^{\alpha}}} } \xrightarrow[R \to \infty]{} \esssup_{y} H (f((x-y)/a), g(y/a)).
\end{align*}
On the other hand let $\varepsilon>0$. Consider 
$$
A_{\varepsilon} = \{ y : H(f((x-y)/a), g(y/a)) >  \esssup_{y} H (f((x-y)/a), g(y/a)) - \varepsilon\}.
$$
 Let $N$ be a sufficiently large number such that $|A_{\varepsilon} \cap B(0,N)| > 0$. Here $B(0,N)$ denotes the ball centered at the origin with radius $N$.  Then 
\begin{align*}
&\varphi_{R}(x) \geq \left( \gamma_{p,q,x}(A_{\varepsilon} \cap B(0,N))\right)^{\frac{1}{R-R^{\alpha}}} (\esssup_{y} H (f((x-y)/a), g(y/a)) - \varepsilon)^{\frac{R}{R-R^{\alpha}}} \\
&\xrightarrow[R \to \infty]{} \esssup_{y} H (f((x-y)/a), g(y/a)) - \varepsilon.
\end{align*}
The last passage follows from the fact that the power in the exponent (\ref{measure}) is of order $R^{\beta}$ where $\beta<1$.  Since $\varepsilon$ is arbitrary we obtain the pointwise convergence for $\varphi_{R}(x)$. 
 
Finally, since $\varphi_{R}(x)$ are uniformly  bounded the Lebesgue dominated convergence theorem implies  
\begin{align*}
\int_{\mathbb{R}^{n}} \varphi_{R}(x) d\gamma(x) \xrightarrow[R \to \infty]{}  \int_{\mathbb{R}^{n}} \esssup_{y} \;H\left( f\left( \frac{x-y}{a}\right), g\left(\frac{y}{b}\right) \right)d\gamma(x).
\end{align*}
It remains to prove the lemma. 
\begin{proof}
Take an arbitrary Borel measurable  $\varphi$ such that $\delta' > \varphi > 1/\delta' >0$ for some $\delta'>0$.  Let $a(R) = R-R^{\alpha}$. First,  we  show that 

 \begin{align}\label{pirveli}
 &\int_{\mathbb{R}^{n}} \int_{\mathbb{R}^{n}}H^{R}\left(f\left(\frac{x-y}{a}\right), g\left(\frac{y}{b}\right) \right) \varphi^{1-a(R)}(x) d\gamma_{p,q,x}(y) d\gamma(x)   \geq  \nonumber \\
 &H^{R}\left( \int_{\mathbb{R}^{n}} f\left( x\sqrt{1+ \frac{1}{a^{2}R^{\beta}}  }\right) d\gamma, \int_{\mathbb{R}^{n}} g\left( x\sqrt{1+\frac{1}{b^{2} R^{\beta}} }\right) d\gamma \right) \cdot \left(\int_{\mathbb{R}^{n}} \varphi d\gamma  \right)^{1-a(R)}.
 \end{align}
We should apply Theorem~A. In order to do that notice that 
 \begin{align}\label{medef}
 d\gamma_{p,q,x}(y) d\gamma_{n}(x)=  \frac{e^{-\langle C^{-1} \vec{x}, \vec{x} \rangle/2}}{\sqrt{(2\pi)^{2n}}} \det (C^{-1/2})dxdy := \gamma_{C}(\vec{x})dxdy,
 \end{align}
 where $\vec{x}=(x,y)$ and 
 \begin{align}\label{matrixsi}
 C^{-1} =\begin{pmatrix}
 1+\frac{q^{2}}{p} & q \\
 q & p
 \end{pmatrix} \otimes I_{n \times n} = \tilde{C}^{-1} \otimes I_{n \times n},
 \end{align}
 where $I_{n \times n}$ is $n \times n$ identity matrix. Clearly $C>0$. Set 
 \begin{align*}
 A_{1} := \begin{pmatrix}
 \frac{1}{a}\\
 -\frac{1}{a}
 \end{pmatrix} \otimes I_{n\times n} := a_{1} \otimes I_{n\times n}, \quad
   A_{2} := \begin{pmatrix}
 0\\
 \frac{1}{b}
 \end{pmatrix} \otimes I_{n\times n} := a_{2} \otimes I_{n\times n}, \quad 
   A_{3} := \begin{pmatrix}
 1\\
 0
 \end{pmatrix} \otimes I_{n\times n} := a_{3} \otimes I_{n\times n}.
 \end{align*}
 Let $A:=(A_{1}, A_{2}. A_{3})$ be $2n \times 3n$ matrix. Notice that $A = \tilde{A} \otimes I_{n\times n}$ where $\tilde{A} = (a_{1},a_{2},a_{3})$ is $2 \times 3$ matrix.  
 
  Next, we notice that in this case
  \begin{align*}
  C = \begin{pmatrix}
  1 & -\frac{q}{p}\\
  -\frac{q}{p} & \frac{1}{p}+\frac{q^{2}}{p^{2}}
  \end{pmatrix}\otimes I_{n\times n} = \tilde{C} \otimes I_{n\times n}.
  \end{align*}
Therefore 
\begin{align*}
& A_{3}^{*}CA_{3} = \langle \tilde{C}a_{3},a_{3} \rangle \otimes I_{n \times n}  =I_{n\times n};\\
&A_{2}^{*} CA_{2} = \langle \tilde{C}a_{2}, a_{2} \rangle \otimes I_{n \times n}  = \frac{1}{b^{2}}\left[ \frac{1}{p}+\frac{q^{2}}{p^{2}}\right] \otimes I_{n\times n}; \\
&A_{1}^{*}CA_{1} = \langle \tilde{C}a_{1}, a_{1} \rangle =  \frac{1}{a^{2}}\left[1+\frac{1}{p}+2\cdot \frac{q}{p}+\frac{q^{2}}{p^{2}} \right] \otimes I_{n\times n}.
\end{align*}
 
 By Theorem~A the desired inequality (\ref{pirveli}) holds, namely, 
 \begin{align}
 &\int_{\mathbb{R}^{2n}} H^{R}(f(\vec{x}\,A_{1}), g(\vec{x}\, A_{2})) \varphi^{1-a(R)}(\vec{x}\, A_{3}) d\gamma_{C}(\vec{x}) \geq  \nonumber\\
 &H^{R}\left( \int_{\mathbb{R}^{n}} f(x\, \sqrt{A_{1}^{*}CA_{1}})d\gamma, \int_{\mathbb{R}^{n}} g (x\,\sqrt{A_{2}^{*}CA_{2}})d\gamma\right) \cdot \left(\int_{\mathbb{R}^{n}} \varphi(x\, \sqrt{A_{3}^{*}CA_{3}}) d\gamma  \right)^{1-a(R)} \label{chveni}
 \end{align}
  if and only if $A^{*}CA \bullet \mathrm{Hess}\, B = \{ A_{i}^{*}CA_{j} \partial_{ij} B \}_{i,j=1}^{n} \geq 0$ where  $B(x,y,z)=H^{R}(x,y)z^{1-a(R)}$ is given on $I\times J \times [\delta',1/\delta']$.
  

Denote  $\varepsilon := R^{-\beta}$, and lets think of it   as a sufficiently small number.
 We  remind that $p=\varepsilon^{-1}$, and $q=-b\varepsilon^{-1}$ if $a=|1-b|$ and $q=b\varepsilon^{-1}$ if $a=1+b$.   Then $\langle \tilde{C}a_{2}, a_{2}\rangle =1+ \frac{\varepsilon}{b^{2}}$ and $\langle \tilde{C} a_{1}, a_{1} \rangle =1+\frac{\varepsilon}{a^{2}}$. 

First we consider the case when $a+b=1$. The remaining cases are similar. 
We obtain 
\begin{align}\label{matrixA}
C = \begin{pmatrix}
1 & b \\
b & b^{2}+\varepsilon
\end{pmatrix} \otimes I_{n \times n} \quad \text{and,} \quad 
A^{*}CA = \begin{pmatrix}
1+\frac{\varepsilon}{a^{2}} & 1-\frac{\varepsilon}{ab} & 1 \\
1-\frac{\varepsilon}{ab}  & 1+\frac{\varepsilon}{b^{2}} & 1\\
1 & 1 & 1
\end{pmatrix} \otimes I_{n \times n} = \tilde{A}^{*}\tilde{C}\tilde{A} \otimes I_{n \times n}. 
\end{align}
For $B(x,y,z)=H^{R}(x,y)z^{1-a(R)}$ we have 
\begin{align*}
\mathrm{Hess}\; B =R H^{R-2} z^{1-a(R)} \times 
\begin{pmatrix}
H_{xx}H +(R-1)H_{x}^{2} &H_{xy}H+(R-1)H_{x}H_{y} &(1-a(R))H_{x}H z^{-1}\\
H_{xy}H+(R-1)H_{x}H_{y} &H_{yy}H+(R-1)H_{y}^{2} & (1-a(R))H_{y}Hz^{-1}\\
(1-a(R))H_{x}Hz^{-1} &  (1-a(R))H_{y}Hz^{-1} &\frac{a(R)(a(R)-1)}{R}H^{2}z^{-2}
\end{pmatrix}.
\end{align*}
 We have $\mathrm{Hess}\; B = R H^{R-2} z^{1-a(R)} \cdot S TS^{*}$ where $S$ is a diagonal matrix with entries $1$, $1$ and $(1-a(R))H z^{-1}$ on the diagonal, and 
\begin{align*}
T = 
\begin{pmatrix}
H_{xx}H +(R-1)H_{x}^{2} &H_{xy}H+(R-1)H_{x}H_{y} &H_{x}\\
H_{xy}H+(R-1)H_{x}H_{y} &H_{yy}H +(R-1)H_{y}^{2} & H_{y}\\
H_{x} &  H_{y} &\frac{a(R)}{R(a(R)-1)}
\end{pmatrix}. 
\end{align*}
Thus $A^{*}CA \bullet \mathrm{Hess}\; B \geq 0$ if and only if $\tilde{A}^{*}\tilde{C}\tilde{A} \bullet T \geq 0$. If we set $R-1=N$  and $\frac{a(R)}{R(a(R)-1)}=\frac{1}{M}$ then we have 
\begin{align*}
\tilde{A}^{*}\tilde{C}\tilde{A} \bullet T = 
\begin{pmatrix}
(H_{xx}H +N\cdot H_{x}^{2}) \left(1+\frac{\varepsilon}{a^{2}} \right) &(H_{xy}H+N\cdot H_{x}H_{y})\left(1-\frac{\varepsilon}{ab} \right) &H_{x}\\
(H_{xy}H+N \cdot H_{x}H_{y})\left(1-\frac{\varepsilon}{ab} \right)  &(H_{yy}H +N\cdot H_{y}^{2}) \left(1+\frac{\varepsilon}{b^{2}} \right) & H_{y}\\
H_{x} &  H_{y} &\frac{1}{M}
\end{pmatrix}.
\end{align*}
By choosing $N$ to be sufficiently large 
we will make the diagonal entries positive. Notice that such choice is possible because $H \in C^{2}(\Omega)$ and $H_{x}, H_{y} \neq 0$.  Let us investigate the sign of $2\times 2$ leading minor. We have 
\begin{align}\label{first-minor}
&\det 
\begin{pmatrix}
(H_{xx}H +N\cdot H_{x}^{2}) \left(1+\frac{\varepsilon}{a^{2}} \right) &(H_{xy}H+N\cdot H_{x}H_{y})\left(1-\frac{\varepsilon}{ab} \right)\\
(H_{xy}H+N \cdot H_{x}H_{y})\left(1-\frac{\varepsilon}{ab} \right)  &(H_{yy}H +N\cdot H_{y}^{2}) \left(1+\frac{\varepsilon}{b^{2}} \right)
\end{pmatrix}=
\varepsilon N^{2} H_{x}^{2}H_{y}^{2}\left(\frac{a+b}{ab}\right)^{2}+O(N),
\end{align}
where $O(N) \leq N \cdot C_{1}$ where $C_{1}=C_{1}(H)$ is some absolute constant.  Thus we  see that choosing $N \varepsilon  > C_{2}$ for some arbitrary large absolute $C_{2}=C_{2}(H)$  (we remind that $\varepsilon = R^{-\beta}$ and  $N \sim R$)  the determinant of the minor will be positive. For the next $2 \times 2$ minor we have 
\begin{align}\label{second-minor}
&\det 
\begin{pmatrix}
(H_{yy}H +N\cdot H_{y}^{2}) \left(1+\frac{\varepsilon}{b^{2}} \right) & H_{y}\\
H_{y} &\frac{1}{M}
\end{pmatrix}=H_{y}^{2}\left(\frac{N}{M}-1 +\frac{N\varepsilon}{M b^{2}} \right) + \frac{H_{yy} H}{M}.
\end{align}
Notice that $N-M=\frac{R}{a(R)}-1$ and since $a(R)<R$ we have $H_{y}^{2}\left(\frac{N}{M}-1 +\frac{N\varepsilon}{M b^{2}} \right) + \frac{H_{yy} H}{M} \geq \frac{1}{M}(H_{y}^{2}N\varepsilon+H_{yy}H)$ and the last expression is nonnegative if $\varepsilon N > C_{3}$ for some large absolute $C_{3} = C_{3}(H)$. In the similar way we obtain that all $2 \times 2$ minors are nonnegative provided that $N \varepsilon$ is sufficiently large.

So it remains to  check the sign of $\det(\tilde{A}^{*}\tilde{C}\tilde{A} \bullet T)$.  We have 
\begin{align}
&a^{2} b^{2} M \det(\tilde{A}^{*}\tilde{C}\tilde{A} \bullet T) = N \varepsilon H (a^{2} H_{xx} H_{y}^{2}+b^{2}H_{yy}H_{x}^{2}+(1-a^{2}-b^{2})H_{xy}H_{x}H_{y} )+\label{third-minor} \\
&(N-M)\left[\frac{N \varepsilon}{2} H_{x}^{2} H_{y}^{2}(a+b)^{2}+H(a^{2}b^{2}+ \frac{\varepsilon^{2} N}{N-M})(H_{xx}H_{y}^{2}+H_{yy}H_{x}^{2}-2H_{xy}H_{x}H_{y}) \right.+ \nonumber \\
&\left.\varepsilon H (b^{2}H_{xx}H_{y}^{2}+a^{2}H_{yy}H_{x}^{2}+2H_{xy}H_{x}H_{y}ab)\right]+(N-M)\frac{N \varepsilon}{2} H_{x}^{2} H_{y}^{2}(a+b)^{2}+\nonumber \\
&H^{2}(a^{2}b^{2}+\varepsilon^{2})(H_{xx}H_{yy}-H_{xy}^{2})+\varepsilon H^{2}(H_{xx}H_{yy}(a^{2}+b^{2})+2ab H_{xy}^{2}). \nonumber
\end{align}
Notice that the first term, i.e., $N \varepsilon H (a^{2} H_{xx} H_{y}^{2}+b^{2}H_{yy}H_{x}^{2}+(1-a^{2}-b^{2})H_{xy}H_{x}H_{y})$ is nonnegative by the condition of the theorem. For the rest of the terms we notice  that if we set $a(R)=R-R^{\alpha}$ and $\varepsilon=R^{-\beta}$  for any $1>\alpha>\beta>0$, $\alpha+\beta>1$, then we obtain that for sufficiently large $R$ we have $N-M \approx R^{\alpha-1}$, 
\begin{align*}
N\varepsilon \approx R^{1-\beta} \to \infty, \quad \frac{\varepsilon}{N-M}\approx R^{1-\alpha-\beta} \to 0 \quad \text{and} \quad (N-M)N \varepsilon \approx R^{\alpha-\beta} \to \infty.
\end{align*}
Therefore the second term $\frac{N\varepsilon}{2}H_{x}^{2}H_{y}^{2}$ will dominate the rest of the terms, and in the last terms we notice that $(N-M)\frac{N \varepsilon }{2} H_{x}^{2} H_{y}^{2}$ will dominate all the bounded terms. 

For the remaining cases when $a=b+1$ and $a=b-1$ we have $\tilde{A}^{*}\tilde{C}\tilde{A}=$
\begin{align*}
  \begin{pmatrix}
1+\frac{\varepsilon}{a^{2}} & -1-\frac{\varepsilon}{ab} & -1 \\
-1-\frac{\varepsilon}{ab}  & 1+\frac{\varepsilon}{b^{2}} & 1\\
-1 & 1 & 1
\end{pmatrix} \quad \text{if}\quad  a=b-1; \quad 
\begin{pmatrix}
1+\frac{\varepsilon}{a^{2}} & -1-\frac{\varepsilon}{ab} & 1 \\
-1-\frac{\varepsilon}{ab}  & 1+\frac{\varepsilon}{b^{2}} & -1\\
1 & -1 & 1
\end{pmatrix} \quad \text{if} \quad a=b+1.
\end{align*}
In both of the cases there is a diagonal matrix $D$ having entries $\pm 1$ on the diagonal such that 
\begin{align}\label{diagonal}
D\tilde{A}^{*}\tilde{C}\tilde{A}D = \begin{pmatrix}
1+\frac{\varepsilon}{a^{2}} & 1+\frac{\varepsilon}{ab} & 1 \\
1+\frac{\varepsilon}{ab}  & 1+\frac{\varepsilon}{b^{2}} & 1\\
1 & 1 & 1
\end{pmatrix}.
\end{align} 
Notice that (\ref{diagonal}) is the same as $\tilde{A}^{*}\tilde{C}\tilde{A}$ in (\ref{matrixA}) except $b$ has switched the sign. Formulas (\ref{first-minor}), (\ref{second-minor}) and (\ref{third-minor}) are still valid if we switch the sign of $b$ to $-b$. The rest of the discussions proceed without any changes.

Finally in order to obtain (\ref{mq1}) we take infimum of the left hand side of (\ref{pirveli}) over all positive $\varphi$ such that $\int \varphi d\gamma_{n}=1$. Indeed, for the convenience of the reader let us mention  the following classical result.
\begin{lemma}\label{classical}
Let $s>1$  and $t<0$ be such that $s + t=1$ then for any positive bounded  $F,G$ we have 
\begin{align}\label{jensen}
\int F^{s} G^{t} d\mu \geq \left( \int F d\mu \right)^{s} \left( \int G d\mu \right)^{t}.
\end{align}
The equality holds if $F=\lambda G$ for a constant $\lambda >0$. 
\end{lemma}
\begin{proof}
Indeed, notice that $B(x,y)=x^{s}y^{t}$ is  a 1-homogeneous convex function for $x,y >0$. Therefore (\ref{jensen}) follows from the Jensen's inequality. 
\end{proof}

In case of (\ref{pirveli}) we take $t = 1-a(R)$ and $s=a(R)$. Taking infimum over all positive and bounded $\varphi$ with $\int \varphi d\gamma_{n}=1$ and finally rising the  obtained inequality to the power $1/a(R)$ we obtain (\ref{mq1}). In fact the infimum is attained on the following function
\begin{align}\label{optimizer-1}
\varphi(x)= m\cdot \left(\int_{\mathbb{R}^{n}} H^{R}\left(f \left(\frac{x-y}{a}\right),g\left(\frac{y}{b}\right) \right) d\gamma_{p,q,x}(y) \right)^{1/a(R)}
\end{align}
where the constant $m$ is chosen so that $\int \varphi  d\gamma_{n}=1$. Clearly such optimizer satisfies $\delta'< \varphi \leq 1/\delta'$ for some nonzero $\delta'>0$ because $H$ is bounded and $H > \delta$.   This finishes the proof in the parabolic case. 
\end{proof}


\subsubsection{Elliptic case}
In this subsection we consider the following case 
\begin{align*}
\left(\frac{1-a^{2}-b^{2}}{2ab} \right)^{2} <1.
\end{align*} 

\begin{lemma}\label{lem-7}
There exist positive constants $c=c(H)>0$ and $R_{0}=R_{0}(H)$ such that for any $R>R_{0}$ we have 
\begin{align*}
\int_{\mathbb{R}^{n}}\left( \int_{\mathbb{R}^{n}} H^{R}\left(f\left(\frac{x-y}{a}\right), g\left(\frac{y}{b}\right) \right)d\gamma_{p,q,x}(y)\right)^{\frac{1}{R-c}}d\gamma_{n}(x) \geq  \nonumber  H^{\frac{R}{R-c}}\left(\int_{\mathbb{R}^{n}} f \left( x\right)d\gamma(x),\int_{\mathbb{R}^{n}}g\left(x  \right)d\gamma(x)\right), 
\end{align*}
where $d\gamma_{p,q,x}(y)$ is defined as in (\ref{measure}) and  
\begin{align}\label{choiceofp}
p=\frac{4}{(1-(a-b)^{2})((a+b)^{2}-1)} \quad \text{and} \quad q =  \frac{2(a^{2}-b^{2}-1)}{(1-(a-b)^{2})((a+b)^{2}-1)}.
\end{align}
\end{lemma}

As before using Lemma~\ref{classical} it is enough to prove (\ref{chveni}) for all  bounded,   positive and  uniformly separated from zero $\varphi$ where $a(R)=R-c$, $c$ will be determined later.

Notice that (\ref{choiceofp}) implies that
\begin{align*}
\langle \tilde{C} \tilde{a}_{3}, \tilde{a}_{3} \rangle =1; \quad \langle \tilde{C}\tilde{a}_{2}, \tilde{a}_{2} \rangle = \frac{1}{b^{2}}\left[ \frac{1}{p}+\frac{q^{2}}{p^{2}}\right]=1; \quad 
\langle \tilde{C}\tilde{a}_{1}, \tilde{a}_{1} \rangle = \frac{1}{a^{2}}\left[1+\frac{1}{p}+2\cdot \frac{q}{p}+\frac{q^{2}}{p^{2}} \right]=1.
\end{align*}
We have
\begin{align*}
\tilde{C}=
\begin{pmatrix}
1 & -\frac{q}{p}\\
-\frac{q}{p} & \frac{1}{p}+\frac{q^{2}}{p^{2}}
\end{pmatrix}=
\begin{pmatrix}
1 & \frac{1-a^{2}+b^{2}}{2}\\
\frac{1-a^{2}+b^{2}}{2} & b^2
\end{pmatrix}.
\end{align*}
Notice also that  $C=\tilde{C}\otimes I_{n \times n}$ is positive-definite if and only if $|a-b|<1$ and $a+b>1$. 
Let $A=(A_{1}, A_{2}, A_{3})$ be the same matrix as before. By Theorem~A,  the inequality 

 \begin{align*}
 \int_{\mathbb{R}^{2n}} H^{R}(f(\vec{x}\, A_{1}), g(\vec{x} \, A_{2})) \varphi^{1-a(R)}(\vec{x}\, A_{3}) d\gamma_{C}(\vec{x}) \geq 
 H^{R}\left( \int_{\mathbb{R}^{n}} f(x)d\gamma, \int_{\mathbb{R}^{n}} g (x)d\gamma\right) \cdot \left(\int_{\mathbb{R}^{n}} \varphi(x ) d\gamma  \right)^{1-a(R)} 
 \end{align*}
holds if and only if $A^{*}CA \bullet \mathrm{Hess}\; B \geq 0$ where again $B(x,y,z)=H^{R}(x,y)z^{1-a(R)}$. We have 
\begin{align*}
A^{*}CA  = 
\begin{pmatrix}
1 & \frac{1-a^{2}-b^{2}}{2ab} & \frac{1+a^{2}-b^{2}}{2a}\\
\frac{1-a^{2}-b^{2}}{2ab} & 1 & \frac{1-a^{2}+b^{2}}{2b}\\
 \frac{1+a^{2}-b^{2}}{2a} &  \frac{1-a^{2}+b^{2}}{2b} & 1
\end{pmatrix}\otimes I_{n \times n}.
\end{align*}
Notice that as before it is enough to check positive definiteness of the following matrix 
\begin{align*}
\tilde{A}^{*}\tilde{C}\tilde{A} \bullet T = 
\begin{pmatrix}
H_{xx}H +N\cdot H_{x}^{2} &(H_{xy}H+N\cdot H_{x}H_{y})\left(\frac{1-a^{2}-b^{2}}{2ab} \right) &H_{x}\left(\frac{1+a^{2}-b^{2}}{2a} \right)\\
(H_{xy}H+N \cdot H_{x}H_{y})\left(\frac{1-a^{2}-b^{2}}{2ab} \right)  &H_{yy}H +N\cdot H_{y}^{2} & H_{y}\left(\frac{1-a^{2}+b^{2}}{2b} \right)\\
H_{x}\left(\frac{1+a^{2}-b^{2}}{2a} \right) &  H_{y}\left(\frac{1-a^{2}+b^{2}}{2b} \right) &\frac{1}{M}
\end{pmatrix},
\end{align*}
where $R-1=N$ and $\frac{a(R)}{R(a(R)-1)}=\frac{1}{M}$. If $R$ is sufficiently large then all diagonal entries are positive. One can notice that all principal  $2 \times 2$ minors have positive determinant provided that $R$ is sufficiently large and  $R>a(R)$. This follows from the fact that 
\begin{align*}
1-\left(\frac{1-a^{2}-b^{2}}{2ab}\right)^{2}&=\frac{((a+b)^{2}-1)(1-(a-b)^{2})}{4a^{2}b^{2}}>0, \\
\quad 1-\left(\frac{1-a^{2}+b^{2}}{2b}\right)^{2}&=\frac{((a+b)^{2}-1)(1-(a-b)^{2})}{4b^{2}}>0,
\end{align*}
 and $N-M = \frac{R}{a(R)}-1>0$. 
 
So it remains to check the sign of $\det(\tilde{A}^{*}\tilde{C}\tilde{A} \bullet T)$. We have 
\begin{align*}
&4a^{2}b^{2}M \det(\tilde{A}^{*}\tilde{C}\tilde{A} \bullet T)=MH(1-(a-b)^{2})((a+b)^{2}-1)[a^{2}H_{xx}H_{y}^{2}+(1-a^{2}-b^{2})H_{xy}H_{x}H_{y}+b^{2}H_{yy}H_{x}^{2}]+\\
&(N-M)\left[\frac{1}{2}N H_{x}^{2}H_{y}^{2}(1-(a-b)^{2})((a+b)^{2}-1)+ H(4a^{2}b^{2}(H_{xx}H_{y}^{2}+H_{yy}H_{x}^{2})-2H_{xy}H_{x}H_{y}(1-a^{2}-b^{2})^{2})\right]\\
&\frac{1}{2}(N-M)N H_{x}^{2}H_{y}^{2}(1-(a-b)^{2})((a+b)^{2}-1)+4a^{2}b^{2}H^{2}\left(H_{xx}H_{yy}-H_{xy}^{2}\left( \frac{1-a^{2}-b^{2}}{2ab}\right)^{2} \right)=\\
&=M I_{1}+(N-M)I_{2}+I_{3}+I_{4}.
\end{align*}

Notice that the first term $I_{1} \geq 0$ by (\ref{numbers}). The second term $I_{2}$ contains a factor of the form $\frac{1}{2}NH_{x}^{2}H_{y}^{2}(1-(a-b)^2)((a-b)^{2}-1)$ which will dominate the remaining subterms as $N \to \infty$. Finally the  sum of the last two terms $I_{3}+I_{4}$ will be positive provided that 
\begin{align}\label{final}
(N-M)N \geq  - \frac{8a^{2}b^{2}H^{2}\left(H_{xx}H_{yy}-H_{xy}^{2}\left( \frac{1-a^{2}-b^{2}}{2ab}\right)^{2} \right)}{H_{x}^{2}H_{y}^{2}(1-(a-b)^{2})((a+b)^{2}-1)}.
\end{align} 
The last inequality holds  provided that $c$ is sufficiently large number. Indeed, notice that if $R>R_{0}$ for some large $R_{0}>0$ then    $(N-M)N = \frac{R-1}{a(R)} (R-a(R)) > c$. On the other hand the right hand side of (\ref{final}) is bounded.  This finishes the proof of the lemma.

\section{Applications}\label{app}

\subsection{How to solve PDE}
In this section we describe how to find solutions of the following PDE
\begin{align}\label{EDS0}
a^{2} \frac{H_{xx}}{H_{x}^{2}} + (1-a^{2}-b^{2})\frac{H_{xy}}{H_{x} H_{y}} + b^{2} \frac{H_{yy}}{H_{y}^{2}}=0.
\end{align}

For simplicity we will  stick to the case when $a=b= \frac{1}{2}$, however,  our arguments can be extended to an arbitrary $a,b>0$ without any difficulties. 
\begin{proposition}\label{EDS10}
Let $V(s,t)$ be a smooth function which satisfies the heat equation $V_{ss}=V_{t}$ in a simply connected domain $\Lambda \subset \mathbb{R}^{2}$. Assume that 
\begin{align}\label{difeo01}
(s,t) \to \left( \frac{e^{-s-t}}{2} (V_{t}+V_{s}), \frac{e^{s-t}}{2}(V_{t}-V_{s})\right)
\end{align}
is a $C^{\infty}$ diffeomorphism from $\Lambda$ onto $\mathrm{int}(\Omega)=\mathrm{int}(I\times J)$. Then the smooth function $H(x,y)$ parametrized as 
\begin{align}\label{parametrizacia}
H\left(\frac{e^{-s-t}}{2} (V_{t}+V_{s}), \frac{e^{s-t}}{2}(V_{t}-V_{s}) \right) = V_{t}-V
\end{align}
solves PDE
\begin{align*}
\frac{H_{xx}}{H_{x}^{2}} + 2\frac{H_{xy}}{H_{x} H_{y}} +  \frac{H_{yy}}{H_{y}^{2}}=0,
\end{align*}
and it has the property that $H_{x}, H_{y}> 0$.
\end{proposition}
\begin{proof}
The conclusion of the proposition can be checked by a straightforward computation, but let us explain it in details how the argument works. 
First we linearize (\ref{EDS0}). For now let $a,b>0$.  Let $U : \tilde{\Lambda}\subset \mathrm{int}(\mathbb{R}^{2}_{+}) \to \mathbb{R}$ be a smooth function such that  
\begin{align}\label{difeo}
(p,q) \mapsto (U_{p}, U_{q})
\end{align}
 is a smooth diffeomorphism from $\tilde{\Lambda}$ onto $\mathrm{int}(\Omega)$.  Define $H(x,y)$ using the following system of equations 
\begin{align}\label{EDS1}
\begin{cases}
x = U_{p};\\
y=U_{q};\\
H(x,y) = px+qy - U(p,q).
\end{cases}
\end{align}
Since $U$ is a smooth diffeomorphism, we can find smooth functions $p=p(x,y), q=q(x,y)$ such that the first two equations in (\ref{EDS1}) are satisfied. Differentiating the third equation in (\ref{EDS1}) it follows that $H_{x}=p>0, H_{y}=q>0$. Therefore $H_{xx}=p_{x}, H_{yy}=q_{y}$ and  $p_{y}=q_{x}=H_{xy}$. Taking the differential of the first two equations of (\ref{EDS1}) we obtain 
\begin{align}\label{EDS2}
&p_{x} = \frac{U_{qq}}{U_{pp}U_{qq}-U_{pq}^{2}}, \quad q_{y} = \frac{U_{pp}}{U_{pp}U_{qq}-U_{pq}^{2}} \quad \text{and} \quad p_{y} = -\frac{U_{pq}}{U_{pp}U_{qq}-U_{pq}^{2}}.
\end{align}
Notice that since the mapping (\ref{difeo}) is a smooth diffeomorphism we have $U_{pp}U_{qq}-U_{pq}^{2}\neq 0$, therefore the expressions in (\ref{EDS2}) are well defined. Notice that the transformation (\ref{EDS1}) linearizes the Monge--Amp\`ere type PDE (\ref{EDS0}). Indeed,  PDE (\ref{EDS0}) takes the form 
\begin{align}\label{EDS3}
\frac{a^{2}b^{2}}{p^{2}q^{2} (U_{pp}U_{qq}-U_{pq}^{2})}\times \left( \frac{q^{2} U_{qq}}{b^{2}}-\left( \frac{1-a^{2}-b^{2}}{a^{2}b^{2}}\right)pqU_{pq} + \frac{p^{2}U_{pp}}{a^{2}}\right)=0.
\end{align}
Since $p,q,a,b>0$ and $U_{pp}U_{qq}-U_{pq}^{2}\neq 0$ we can ignore the first factor in the left hand side of (\ref{EDS3}). Next, if define $\widetilde{U}(u,v)$ as $U(p,q)=\widetilde{U}(a\ln p, b \ln q)$, then equation (\ref{EDS3}) takes the following standard form 
\begin{align}\label{EDS4}
\widetilde{U}_{uu}+\widetilde{U}_{vv} - \left(\frac{1-a^{2}-b^{2}}{ab}\right)\widetilde{U}_{uv} - \frac{\widetilde{U}_{u}}{a} - \frac{\widetilde{U}_{v}}{b}=0.
\end{align}
Pick numbers $u_{1}, u_{2}, v_{1}, v_{2}$ such that $u_{1}v_{2}-v_{1}u_{2}\neq 0$, and define the function $V(s,t)$ as 
\begin{align*}
\widetilde{U}(u,v)=V(u_{1} u + v_{1}v, u_{2}u+v_{2}v).
\end{align*}
After the direct computations we notice that (\ref{EDS4}) takes the form 
\begin{align}
&\left(u_{1}^{2}+u_{2}^{2} - \frac{1-a^{2}-b^{2}}{ab}u_{1}v_{1} \right)V_{ss} + \left(u_{2}^{2}+v_{2}^{2}-\frac{1-a^{2}-b^{2}}{ab}u_{2}v_{2}\right)V_{tt}+\label{EDS5}\\
&\left(2(u_{1}v_{1}+u_{2}v_{2})+u_{1}v_{2}+u_{2}v_{1} \right)V_{st} - \left(\frac{u_{1}}{a}+\frac{v_{1}}{b}\right)V_{s} - \left(\frac{u_{2}}{a}+\frac{v_{2}}{b}\right)V_{t}=0.\nonumber
\end{align}
Next, if we consider $a=b=1/2$, then we see that choosing $u_{1}=-v_{1}=u_{2}=v_{2}=1$ the equation (\ref{EDS5}) simplifies to the heat equation 
\begin{align}\label{heateq1}
V_{ss}-V_{t}=0.
\end{align}
Tracing back to our  change of variables we obtain 
\begin{align}\label{amockda}
U(p,q) = V\left(\frac{1}{2}\ln \left(\frac{p}{q}\right), \frac{1}{2}\ln(pq)\right). 
\end{align}
Therefore, the system of equations  (\ref{EDS1}), namely, $H(U_{p}, U_{q}) = px+qy-U(p,q)$ transform to (\ref{parametrizacia}), and the fact that $(p,q)\mapsto (U_{p}, U_{q})$ is diffeomorphism implies  that the mapping (\ref{difeo01}) is a smooth diffeomorphism. 
\end{proof}

Such a systematic approach to Monge--Amp\'ere type PDEs the reader can find in a more comprehensive theory of Exterior Differential Systems of Bryant--Griffiths, see, for example, ~\cite{BryantG}. 

\begin{remark}\label{hypokrat}
In general, for an arbitrary $a,b>0$, if $\frac{|1-a^{2}-b^{2}|}{2ab}<1$, then (\ref{EDS5}), after a suitable change of variables, reduces to an elliptic equation, namely, Laplacian eigenvalue problem.  If $\frac{|1-a^{2}-b^{2}|}{2ab}=1$, then (\ref{EDS5}) reduces to Parabolic equation, namely, heat equation. 
\end{remark}

\begin{remark}
For the mapping (\ref{difeo01}) to be smooth diffeomorphism we should assume that the determinant of its Jacobian  matrix is nonzero. Using $V_{ss}=V_{t}$ the determinant takes the form $-\frac{e^{-2t}}{2}(V_{sss}-V_{s})^{2}$, and we obtain the necessary condition  $V_{s}\neq V_{sss}$. 
\end{remark}

\begin{remark}
If one is only interested with partial differential  inequality (\ref{numbers}) unlike (\ref{EDS0}), then instead of requiring $V_{ss}=V_{t}$ we need only  require that $V_{ss}\leq V_{t}$ and $U_{pp}U_{qq}-U_{pq}^{2}<0$ where $U$ is defined as in (\ref{amockda}). 
\end{remark}

Next, let us illustrate how Proposition~\ref{EDS10} works on the examples. 
\subsection{The Ehrhard function}
Take 
\begin{align*}
V(s) := \left(\frac{e^{s}+e^{-s}}{2}-1\right) \mathbbm{1}_{(0, \infty)}(s).
\end{align*}
We recall that the heat extension $V(s,t)$ of the initial data $V(s)$ can be written as 
\begin{align*}
V(s,t)=\int_{\mathbb{R}}V(s+\sqrt{2t} y) d\gamma_{1}(y) = \frac{1}{2}\left(e^{s+t}\Psi\left(\frac{s}{\sqrt{2t}}+\sqrt{2t}\right)+e^{-s+t}\Psi\left(\frac{s}{\sqrt{2t}}-\sqrt{2t}\right) \right)-\Psi\left(\frac{s}{\sqrt{2t}}\right),
\end{align*}
where $t\geq 0$, and $\Psi(x) = \gamma_{1}((-\infty, x])$ is the Gaussian distribution function. 
After straightforward computations it follows that
\begin{align*}
&\frac{e^{-s-t}}{2}(V_{t}+V_{s})=\frac{1}{2}\Psi\left(\frac{s}{\sqrt{2t}}+\sqrt{2t}\right);\\
&\frac{e^{s-t}}{2} (V_{t}-V_{s})=\frac{1}{2}\Psi\left(\frac{s}{\sqrt{2t}}-\sqrt{2t}\right);\\
&V_{t}-V = \Psi\left(\frac{s}{\sqrt{2t}}\right).
\end{align*} 
Denoting $x = \frac{1}{2} \Psi(\frac{s}{\sqrt{2t}}+\sqrt{2t})$, and $y=\frac{1}{2} \Psi(\frac{s}{\sqrt{2t}}-\sqrt{2t})$ for $(x,y) \in (0,1/2)^{2}$ with $x>y$, and using (\ref{parametrizacia}) we obtain 
\begin{align}\label{vipovet}
H(x,y) = \Psi\left(\frac{\Psi^{-1}(2x)+\Psi^{-1}(2y)}{2}\right) \quad\text{where} \quad x>y \quad{and}\quad  (x,y) \in (0,1/2)^{2}.
\end{align}
Stretching the variables  $\tilde{x} = 2x$, $\tilde{y}=2y$ and extending the definition of $H$ in a natural way to the domain $y>x$ we obtain the Ehrhard function (see Section~\ref{lyapa}). 

Next we consider a more peculiar example.
\subsection{Example with Hermite polynomial}
Take
\begin{align*}
V(s)=s^{2} \quad \text{for}\quad s<-1. 
\end{align*}
Clearly  $V(s,t)=s^{2}+2t$ for $t\in \mathbb{R}$ solves the heat equation with $V(s,0)=s^{2}$. In this case we have 
\begin{align*}
&\frac{e^{-s-t}}{2}(V_{t}+V_{s})=(s+1)e^{-s-t};\\
&\frac{e^{s-t}}{2} (V_{t}-V_{s})=(1-s)e^{s-t};\\
&V_{t}-V = 2-s^{2}-2t.
\end{align*}
Notice that the mapping 
\begin{align*}
(s,t) \mapsto \left((s+1)e^{-s-t}, (1-s)e^{s-t}\right)
\end{align*}
is a smooth diffeomorphism from $(-\infty-1,0)\times \mathbb{R}$ onto $(-\infty,0) \times (0, \infty)$. Indeed, let 
\begin{align}
&(1+s)e^{-s-t}=x<0; \label{eq32}\\ 
&(1-s)e^{s-t}=y>0. \label{eq33}
\end{align}
Then $\ell(s):= \frac{1-s}{1+s} e^{2s}$ maps $(-\infty,-1)$  onto $(-\infty,0)$, and it is decreasing. Let $r(s)$ be its inverse map. It follows from (\ref{eq32}) and (\ref{eq33}) that 
\begin{align*}
&s = r\left(y/x\right);\\
&t = \frac{1}{2}\ln \left(\frac{1-r^{2}(y/x)}{xy}\right).
\end{align*}
Thus we obtain 
\begin{align*}
H(x,y) =V_{t}-V =  2- r^{2}(y/x) - \ln \left(\frac{1-r^{2}(y/x)}{xy}\right), 
\end{align*}
with $H \in C^{3}((-\infty,0)\times (0, \infty))$, and $H$ satisfies (\ref{EDS0}). Therefore by Thoerem~\ref{mth1} we obtain inequality 
\begin{align}\label{gakotrda}
&\int_{\mathbb{R}^{n}} \mathrm{ess\,inf}_{y}\;  r^{2}\left( \frac{f(2(x-y))}{g(2y)}\right)+\ln \left(r^{2}\left( \frac{f(2(x-y))}{g(2y)}\right)-1\right) - \ln |f(2(x-y))g(2y)|\, d\gamma(x) \leq \\
&r^{2}\left(\frac{\int_{\mathbb{R}^{n}} f d\gamma}{\int_{\mathbb{R}^{n}} g d\gamma}\right)+ \ln \left(r^{2}\left(\frac{\int_{\mathbb{R}^{n}} f d\gamma}{\int_{\mathbb{R}^{n}} g d\gamma}\right)-1\right) - \ln \left|\int_{\mathbb{R}^{n}} f d\gamma \int_{\mathbb{R}^{n}} g d\gamma \right| \nonumber
\end{align}
for all bounded Borel functions $f<0$, $g>0$ and uniformly separated from zero. We do not know if the estimate (\ref{gakotrda}) can be obtained from the Ehrhard inequality. 

\bigskip 

Sometimes one can try to guess a function $H(x,y)$ which would satisfy (\ref{numbers}). Let us show how this guess works.

Next, we will assume that $a,b>0$, $a+b\geq 1$ and  $|a-b|\leq 1$.   For any real $p>0$, and any Borel function $f$ we define 
\begin{align*}
\| f \|_{L^{p}(d\gamma)} = \left(\int_{\mathbb{R}^{n}} |f|^{p} d\gamma \right)^{1/p}. 
\end{align*}

\subsection{Young's functions}
\begin{corollary}\label{ee-1}
 Let $p, q>0$. The following inequality holds
\begin{align}\label{vse-1}
\int_{\mathbb{R}^{n}} \esssup_{y\in \mathbb{R}^{n}}\; f^{p}\left( \frac{x-y}{a}\right) g^{q}\left(\frac{y}{b}\right)  d\gamma \geq \left( \int_{\mathbb{R}^{n}} fd\gamma \right)^{p} \left(\int_{\mathbb{R}^{n}} g d\gamma \right)^{q} 
\end{align}
for all nonnegative Borel functions $f,g \in L^{1}(d\gamma_{n})$ if and only if $\frac{a^{2}}{p}+\frac{b^{2}}{q}\leq 1$. 
\end{corollary}
  We notice that the case $p=a, q=b$ with $a+b=1$ recovers the Pr\'ekopa--Leindler inequality. 
\begin{proof}

First let us obtain (\ref{vse-1}) for bounded $f$ and $g$ and uniformly separated from zero. Set $H(x,y) = x^{p}y^{q}$  on some bounded closed rectangular domain $\Omega\subset \mathrm{int}(\mathbb{R}^{2}_{+})$.  Then  (\ref{ragaca}) holds  if and only if 
  $\frac{a^{2}}{p}+\frac{b^{2}}{q}\leq 1$.   Indeed, notice that (\ref{numbers}) takes the form 
 \begin{align*}
 a^{2} \frac{H_{xx}}{H_{x}^{2}} + (1-a^{2}-b^{2})\frac{H_{xy}}{H_{x} H_{y}} + b^{2} \frac{H_{yy}}{H_{y}^{2}}=\frac{1}{x^{p}y^{q}}\left( 1-\frac{a^{2}}{p}-\frac{b^{2}}{q} \right) \geq 0.
  \end{align*}
 Thus we obtain (\ref{vse-1}) for bounded functions $f,g$ and uniformly separated from zero, i.e., $f, g \geq \varepsilon$ for some $\varepsilon >0$. The general case  of bounded  $f,g$ follows by considering $\tilde{f}=f + \varepsilon$ and $\tilde{g} = g + \varepsilon$. By sending $\varepsilon \to 0$ and using the dominated convergence theorem we obtain (\ref{vse-1}) for bounded $f,g$ with positive integrals. 
   
   For arbitrary $f$ and $g$ we can approximate   by bounded $f_{n}:=\min\{n, f\} \leq f$ and $g_{n}:=\min\{g, n\} \leq g$ with $f_{n} \to f$ in $L^{1}$ and  $g_{n} \to g$ in $L^{1}$. Since $\esssup_{y} f^{p}((x-y)/a)g(y/b)| \geq \esssup_{y} f_{n}^{p}((x-y)/a)g^{q}_{n}(y/b)$ almost everywhere we obtain the desired result. 
\end{proof}
  
  \vskip1cm
  
\subsection{Minkowski's functions: reverse inequalities}

 \begin{corollary}\label{ee-2}
 Let $p,q, r>0.$ Then 
 \begin{align}\label{minkowski}
\left\|  \esssup_{y\in \mathbb{R}^{n}}\; \left( f\left( \frac{x-y}{a}\right) +  g\left(\frac{y}{b}\right)\right) \right\|_{L^{r}(d\gamma)} \geq \| f\|_{L^{p}(d\gamma)} +\| g\|_{L^{q}(d\gamma)}
\end{align}
for all  nonnegative $f \in L^{p}(d\gamma)$ and $g \in L^{q}(d\gamma)$ if and only if $0<p,\, q \leq 1$ and $r \geq 1-(a\sqrt{1-p}+b\sqrt{1-q})^{2}$. 
 \end{corollary}
 \begin{proof}
 Indeed, consider $H(x,y)=(x^{\frac{1}{p}}+y^{\frac{1}{q}})^{r}$. 
  Let $s=y^{1/q} x^{-1/p}$. Then notice that (\ref{numbers}) takes the form
 \begin{align}
 &\frac{1}{sr(x^{1/p}+y^{1/q})^{r}}\left( a^{2}s^{2}(1-p)+s(b^{2}(1-q)+a^{2}(1-p)-1+r)+b^{2}(1-q) \right)= \label{ex21} \\
 &\frac{1}{sr(x^{1/p}+y^{1/q})^{r}}\left( (as\sqrt{(1-p)}-b\sqrt{1-q})^{2}+s(r-1+(a\sqrt{1-p}+b\sqrt{1-q})^{2}) \right). \label{ex22}
 \end{align}
 For the quantity in (\ref{ex21}) to be nonnegative it is necessary that $p,q \leq 1$. We can assume that $p,q \neq 1$ otherwise the conclusion follows easily. 
Finally (\ref{ex22}) implies that (\ref{ragaca}) holds if and only if $r \geq 1 - (a\sqrt{1-p}+b\sqrt{1-q})^{2}$. Now we consider (\ref{ragaca}) with the test functions $\tilde{f} = f^{p}$,  $\tilde{g} = g^{q}$ and we obtain (\ref{minkowski}). 
 \end{proof}
 
 It is interesting to mention that if $a\sqrt{1-p}+b\sqrt{1-q}=1$ then we can take $r \to 0$  in (\ref{minkowski}), since $\lim\| h\|_{L^{r}} \to \exp(\int \ln  |h| d\mu)$ (assuming $-\infty<\int \ln |h|<\infty)$, we obtain the following corollary
\begin{corollary}\label{ee-3}
Let $p,q>0$ be such that $a\sqrt{1-p}+b\sqrt{1-q}=1$. Then 
 \begin{align*}
\int_{\mathbb{R}^{n}}\;   \esssup_{y}\;  \ln  \left( f\left( \frac{x-y}{a}\right) +  g\left(\frac{y}{b}\right)\right) d\mu  \geq \ln \left( \| f\|_{L^{p}(d\gamma)} +\| g\|_{L^{q}(d\gamma)} \right).
\end{align*}
for all nonnegative $f \in L^{p}(d\gamma)$ and $g \in L^{q}(d\gamma)$. 
\end{corollary}

\bigskip

\subsection{Ehrhard inequality and the Gaussian measure}\label{lyapa}

In what follows $H$ will not belong to the class $C^{3}(\Omega)$. Instead we will only have $H \in C^{3}(\mathrm{int}(\Omega))$ and $H$ is lower-semicontinuous on $\Omega$.  Thus we cannot directly apply Theorem~\ref{mth1}. In order to avoid this obstacle we will slightly modify the functions $H$ and then pass to the limit in (\ref{ragaca}). For example, if $\Omega=[0,1]^{2}$ we will consider auxiliary functions $H_{\varepsilon_{1}, \varepsilon_{2}, \delta_{1}, \delta_{2}}(x,y)=H(\varepsilon_{1}+x \delta_{1}, \varepsilon_{2} + y\delta_{2} )$ for $0<\varepsilon_{1}, \varepsilon_{2}, \delta_{1}, \delta_{2}<1$, and we apply (\ref{ragaca}) to these functions. Finally we just send $\varepsilon_{1}, \varepsilon_{2} \to 0$ and $\delta_{1}, \delta_{2} \to 1$ in the appropriate order.

 Let $\Psi(s) = \int_{-\infty}^{s}  d\gamma$ (this is slightly different notation unlike the classical one $\Phi$). The Ehrhard inequality states that if $a+b  \geq 1$ and $|a-b| \leq 1$ then 
 \begin{align}\label{ehrhard}
 \gamma_{n}(aA+bB) \geq \Psi\left( a \Psi^{-1}(\gamma_{n}(A))+b\Psi^{-1}(\gamma_{n}(B))\right)
 \end{align}
 for all Borel measurable $A,B \subset \mathbb{R}^{n}$ such that the Minkowski sum $aA+bB$ is measurable. The equality is attained in (\ref{ehrhard}) for the half spaces with one containing the other one. 
 
  Inequality (\ref{ehrhard}) was proved by Ehrhard \cite{eh1} when $A$ and $B$ are convex sets under the assumptions that $a+b=1$.  Ehrhard, by developing Gaussian symmetrization method, showed that (\ref{ehrhard}) is enough to prove in the case $n=1$. It was an open problem whether (\ref{ehrhard}) holds for the Borel measurable sets $A$ and $B$ (see \cite{ledoux2}). Latala \cite{lat1} showed that the inequality is true if at least one of the sets is convex (again under the constraints  $a+b=1$).  It was also noticed that the inequality is equivalent to its functional version 
\begin{align}\label{borel-result12}
 \int_{\mathbb{R}^{n}}\sup_{ax+by=t} \Psi\left( a \Psi^{-1}(f(x))+b\Psi^{-1}(g(y)) \right) d\gamma_{n} \geq \Psi\left( a \Psi^{-1}\left(\int f d\gamma_{n} \right)+b\Psi^{-1}\left(\int g d\gamma_{n}\right) \right)
\end{align}
for all smooth $f,g : \mathbb{R}^{n} \to  [\delta,1-\delta]^{2}$ for some $0<\delta<1/2$. 
 Finally,  Borell in his series of papers \cite{bor1,bor2}  using a subtle maximum principle (see Lemma~1 in \cite{bar1} which was later called {\em hill property} in \cite{IV}) obtained (\ref{borel-result12}). Recently,  Ramon~\cite{ramon} gave an elegant  proof of the Ehrhard inequality using $\inf \sup$ representation via stochastic processes for a certain Bellman function. Also recently the author learned that Neeman--Paouris \cite{PN} gave an interpolation proof of the Ehrhard inequality using a more subtle version of  Theorem~A, and they asked a question if one can deduce the Ehrhard inequality using solely Theorem~A (the positive answer wad demonstrated in the previous section). 
 
Ehrhard inequality can be used to find {\em isoperimetric profile} for the Gaussian measure.
 Let $d\mu$ be a probability measure on $\mathbb{R}^{n}$. Let $A_{\varepsilon}$ be an epsilon neighborhood of the set $A$. Set 
\begin{align*}
\mu^{+}(A) = \liminf_{\varepsilon \to 0}\,  \frac{\mu(A_{\varepsilon}) - \mu(A)}{\varepsilon} \quad \text{and} \quad I_{\mu}(p) = \inf_{\mu(A)=p} \mu^{+}(A).
\end{align*}

The function $I_{\mu}(p)$  is called isoperimetric profile of the measure $\mu$.   $I_{\mu}(p)$ measures minimal {\em perimeter} of the set $A$ under the constraint that $\mu(A)=p$ is fixed. One can obtain from (\ref{ehrhard}) that $I_{\gamma_{n}}(p) \geq  \Psi' (\Psi^{-1}(p))$ which is regarded as an infinitesimal version of  $\gamma_{n}(A_{\varepsilon}) \geq \Psi(\Psi^{-1}(p)+\varepsilon)$. A subtle result of Bobkov \cite{bobkov1} asserts that for any even, log-concave measure  $d\mu$ on the real line we have $I_{\mu}(p) = \Phi'(\Phi^{-1}(p))$ where $\Phi(x) = \int_{-\infty}^{x} d\mu$.  We should also  mention that if $d\mu$ is a probability measure with positive distribution function $\Phi$  on the real line then there is a trivial upper bound   
\begin{align*}
\inf_{A, B \subset \mathbb{R}:\; \mu(A)=x, \mu(B)=y} \mu(aA+bB) \leq \Phi(a\Phi^{-1}(x)+b\Phi^{-1}(y)) \quad \text{for all} \quad x, \, y >0.
\end{align*}
The inequality is exhausted by half-lines. If $d\mu$ is a log-concave measure  then we also have a trivial lower bound $\mu(aA+bB) \geq \mu(A)^{a}\mu(B)^{b}$ for $a+b=1$ via the Pr\'ekopa--Leindler inequality. 

These considerations  motivate to the following question:  which measures $d\mu$ satisfy  (\ref{ehrhard}) or (\ref{borel-result12}) with $\Psi$ replaced by $\Phi(x)$, i.e., with a distribution function of  $d\mu$. Further by $d\mu_{n}$ we denote product measure, i.e., $\mu_{n}  = \mu \times \mu \times \ldots \times  \mu$. 

 \begin{theorem}\label{theorem2}
Let $d\mu$ be   a probability measure with positive density function  $\varphi=e^{-V}, V \in C^{2}(\mathbb{R})$ and finite absolute fifth moment. 
Let $\Phi(s) = \int_{-\infty}^{s} d\mu$. Let $a$ and $b$ be some fixed  positive numbers. 
Let 
$$
H(x,y)= \Phi(a \Phi^{-1}(x) + b \Phi^{-1}(y)) \quad \text{on}\quad  [0,1]^{2} \setminus \{(0,1) \cup (1,0)\},
$$
 and set $H(0,1)=H(1,0)=0$.  
\begin{itemize}
\item[(i)]
For the inequality 
 \begin{align}\label{borel-result}
 \int_{\mathbb{R}^{n}}\esssup_{y \in \mathbb{R}^{n}} \, H\left(  f\left( \frac{x-y}{a}\right),  g\left( \frac{y}{b}\right) \right)  d\mu_{n}(x) \geq H\left( \int f d\mu_{n},\int g d\mu_{n} \right)
\end{align}
  to hold for $n =1 $ and  all Borel measurable  $f,g :\mathbb{R}^{n} \to [0,1]$   it is necessary that 
\begin{align}\label{log-prime}
  V'(ax+by)\leq aV'(x)+bV'(y), \quad |1-a^{2}-b^{2}|\leq 2ab, \quad\quad  \int_{\mathbb{R}^{n}} x d\mu =0  \quad \text{if} \quad a+b>1. 
\end{align}
 \item[(ii)] If $d\mu_{n}= d\gamma_{n}$  is the Gaussian measure then (\ref{log-prime}) is necessary and sufficient  for the inequality (\ref{borel-result}) to hold for  any $n\geq 1$, and all Borel $f,g :\mathbb{R}^{n} \to [0,1]$.   
 \end{itemize}
 \end{theorem}
 Further inequality (\ref{borel-result}) we call the Ehrhard inequality. 
\begin{proof} First we prove the necessity part. 
We consider  $H(x,y)$ on  the domain $\Omega_{\delta} = [\delta, 1-\delta]^{2}$ for some $\delta \in (0,1/2)$. Clearly $H \in C^{3}(\Omega_{\delta})$ and in particular (\ref{borel-result}) holds for Borel measurable $f,g : \mathbb{R}^{n}  \to [\delta, 1-\delta]$.  Notice that $H_{x}$ and $H_{y}$ never vanish in $\Omega_{\delta}$. It remains to use the Theorem~\ref{mth1}.  Direct computations show that 
\begin{align*}
 &a^{2} \frac{H_{xx}}{H_{x}^{2}} + (1-a^{2}-b^{2})\frac{H_{xy}}{H_{x} H_{y}} + b^{2} \frac{H_{yy}}{H_{y}^{2}}  = \\
 &\frac{1}{\varphi(a\Phi^{-1}(x)+b\Phi^{-1}(y))}\left[ -a\frac{\varphi'(\Phi^{-1}(x))}{\varphi(\Phi^{-1}(x))}-b\frac{\varphi'(\Phi^{-1}(y))}{\varphi(\Phi^{-1}(y))}+\frac{\varphi'(a\Phi^{-1}(x)+b\Phi^{-1}(y))}{\varphi(a\Phi^{-1}(x)+b\Phi^{-1}(y))}\right].
 \end{align*}
 Therefore  if we introduce new variables $\tilde{x} = \Phi^{-1}(x)$ and $\tilde{y} = \Phi^{-1}(y)$ we see that  by Theorem~\ref{mth1} the necessary condition for  (\ref{borel-result})  is (\ref{log-prime}).  
 
For the sufficiency condition we should introduce an auxiliary  function 
\begin{align}\label{auxiliary}
H_{\varepsilon, \delta}(x,y) = H(\varepsilon+x\delta, \varepsilon+y\delta) \quad \text{for} \quad 0<\varepsilon, \delta<1, \quad \varepsilon+\delta<1.
\end{align}

 Notice that $H_{\varepsilon, \delta} \in C^{3}([0,1]^{2})$ and it satisfies (\ref{numbers}). By Theorem~\ref{mth1} we have (\ref{ragaca}) for $H_{\varepsilon, \delta}$ and $\mu=\gamma_{n}$. We consider 
\begin{align*}
h_{\varepsilon, \delta}(x)  = \esssup_{y \in \mathbb{R}^{n}} \; H_{\varepsilon, \delta }\left(f\left(\frac{x-y}{a}\right), g\left(\frac{y}{b}\right)\right). 
\end{align*}
From $\lim_{\varepsilon \to 0+}\|H_{\varepsilon, \delta} - H_{0, \delta}\|_{C([0,1])}=0$ it follows that $\ h_{\varepsilon, \delta}  \to h_{0,\delta}$ in $L^{1}(d\mu_{n})$. 
 Using the fact that $H$ is increasing in each variable we obtain $h_{0, \delta} \leq h_{0, 1}$. This gives the left hand side of (\ref{borel-result}). For the right hand side notice that  if the point $\left(\int f, \int g \right)$ coincides with $(0,1)$ or $(1,0)$ then there is nothing to prove because $H \geq 0$. In the remaining case when $\left(\int f, \int g \right)$ is the point of continuity of $H$ in $[0,1]^{2}$  we obtain the right hand side of (\ref{borel-result}) by taking the limit. 
\end{proof}

The next corollary says that  in the class of even probability measures on the real line with smooth positive density and finite moments,  the only measures which satisfy the Ehrhard inequality are the Gaussian measures. 
 \begin{corollary}\label{even-measure}
An even probability  measure $d\mu$ with finite absolute fifth moment and the density function  $e^{-V}, V \in C^{2}(\mathbb{R})$ satisfies Ehrhard-type inequality (\ref{borel-result}) with $n=1$ and  some $a,b >0$  if and only if it is the Gaussian measure. 
 \end{corollary}
 \begin{proof}
By (\ref{log-prime}) and the fact that $V'$ is an odd function we obtain
\begin{align*}
-V'(ax+by) = V'(-ax-by) \leq aV'(-x)+bV'(-y) = -aV'(x)-bV'(y).
\end{align*}
Therefore $V'(ax+by)=aV'(x)+bV'(y)$ for all $x,y \in \mathbb{R}$. If we take derivative with respect to $x$ we obtain $aV''(ax+by)=aV''(x)$. Choose $y$ so that $ax+by=0$  then we obtain that $V''(x)=V''(0)$ for all $x \in \mathbb{R}$. Thus $V=cx^{2}+d$ if $a+b>1$ and $V=cx^{2}+kx + d$ if $a+b=1$. Further we will just write $V(x) = cx^{2}+k(a+b-1)x+d$ instead of considering previous cases separately. In both cases $c>0$ because $\int e^{-V} <\infty$. 

On the other hand testing the Ehrhard inequality (\ref{borel-result}) with $d\gamma_{1}$ and test functions $\tilde{f}(x)=f(px+q)$, $\tilde{g}(y)=g(py+q)$ we see that after the change of variables the inequality holds for the probability measures $d\mu=e^{-V}dx$ with $V(x)=e^{cx^{2}+k(a+b-1)x+d}$. 
 \end{proof}
 The following remark was pointed out to us by R.~Lata\l a. 
 \begin{remark}If we drop the assumption of smoothness, namely, $d\mu= e^{-V}, V \in C^{2}(\mathbb{R})$ then Corollary~\ref{even-measure} fails. Indeed, consider $d\mu(x) = \mathbbm{1}_{[-1/2,1/2]}(x)dx$. We are thankful to R.~Lata\l a for pointing out this example. 
 \end{remark}
 
 \bigskip 
 
 It turns out that the measures $e^{-V}$ which satisfy $V'(ax+by)\leq aV'(x)+bV'(y)$ for all $x,y \in \mathbb{R}$ and all $a,b>0$ with $a+b \geq 1$ and $|a-b|\leq1$ have a simple geometrical description. 
 \begin{corollary}\label{last3}
 Let $d\mu$ be a probability measure with the density function  $e^{-V}, V \in C^{2}(\mathbb{R})$ and finite absolute fifth moment. Assume that the Ehrhard inequality (\ref{borel-result}) holds for all $a,b >0$ with $a+b \geq 1$ and $|a-b|\leq 1$. Then $\int_{\mathbb{R}} x e^{-V(x)} dx =0$ and   $V'$ is a convex function. 
 Moreover, there exist constants $c_{\pm}>0$  with $c_{-}\leq c_{+}$ such that $\lim_{x \to \pm \infty}|V'(x)-x c_{\pm}|=0$. 
 \end{corollary}
 
 \begin{proof}
 By Theorem~\ref{theorem2} we have 
 \begin{align}\label{geometry}
 V'(ax+by) \leq aV'(x)+bV'(y)
 \end{align} 
 for all  real $x,y$ and for all positive numbers $a, b$ with $a+b \geq 1$ and $|a-b| \leq 1$. Inequality (\ref{geometry}) has  the following geometrical meaning.  
 Let $\mathrm{epi}\; V' = \{ (x,y) \in \mathbb{R}^{2}\, : \; y \geq V'(x)\}$ be the epigraph of $V'$. Condition (\ref{geometry}) means that $a (x,V'(x))+b (y,V'(y)) \in \mathrm{epi}\; V'$ for all $a+b \geq 1$ and $|a-b|\leq 1$. It follows that the infinity  parallelogram $P$ (see Figure~\ref{fig:pizza}) with sides 
 \begin{align*}
 &s\cdot (x,V'(x))+(1-s)\cdot (y,V'(y)), \quad s \in [0,1];\\
 &(x,V'(x))+s \cdot ((x,V'(x))+(y,V'(y))), \quad s \in [0,\infty)\\
 &(y,V'(y))+s \cdot ((x,V'(x))+(y,V'(y))), \quad s \in [0,\infty)\\ 
 \end{align*}
 belongs to $\mathrm{epi}\, V'$. Since this is true for all $x,y \in \mathbb{R}$ it follows that this can happen if and only if
  $V'$ is convex and $\mathrm{epi}\, V'$ contains all lines  $L$ of the form $s\cdot (x,V'(x)), \;  s \geq 1$ for all $x \in \mathbb{R}$.  Then it follows that there exist  real numbers $c_{\pm}$, $c_{-} \leq c_{+}$ such that $\lim_{x \to \pm \infty}| V(x) - x c_{\pm}|=0$.  Since $\int e^{-V} dx <\infty$ it follows that there exists sufficiently large $p, q>0$ such that $V'(p)>0$ and $V'(-q)<0$.  This implies that $c_{\pm}>0$. 
  
\begin{figure}
\includegraphics[scale=1]{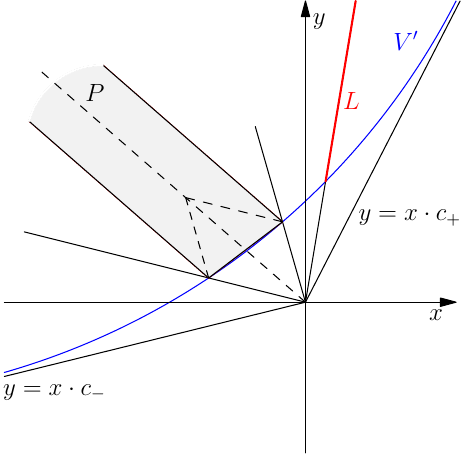}
\caption{Graph of $V'$, parallelogram $P$ and the line $L$}
\label{fig:pizza}
\end{figure}
\end{proof}

One can observe that  $c_{+}=c_{-}$ if and only if $d\mu$ is the Gaussian measure. It would be interesting to see  whether the converse of Corollary~\ref{last3} is also true at least for $a+b=1$, i.e., a probability measure with density $e^{-V}$ and the function $V$ described in Corollary~\ref{last3} satisfies the Ehrhard inequality (\ref{borel-result}) with $n=1$. If this is the case then for such measures we obtain 
$\mu(A_{\varepsilon}) \geq \Phi\left(\Phi^{-1}(\mu(A))+\varepsilon \sqrt{\frac{c_{-}}{c_{+}}}\right)$.

 Next we investigate 1-homogeneous functions which satisfy (\ref{ragaca}). The class of 1-homogeneous functions was studied in a remarkable paper of Borell \cite{Borell2}. One should compare our results of Subsection~\ref{lebeg-hom} with the results of Borell. For the convenience of the reader we have included Borell's theorem in Appendix (see Theorem~B).

  \subsection{Lebesgue measure and 1-homogeneous functions}\label{lebeg-hom}
  In this section we describe all 1-homogeneous functions $H$ which satisfy (\ref{ragaca}). It turns out that they are either convex functions, or the Pr\'ekopa--Leindler type functions (\ref{p-1}), (\ref{p-2}) and (\ref{p-3}). Further we will always  assume that the numbers $a,b>0$ satisfy the constraint  $a+b \geq 1$ and $|a-b|\leq 1$. 
 

\begin{corollary}\label{concavity-gaussian}
Let $H \in C^{3}(\mathrm{int}(\mathbb{R}_{+}^{2}))$ be 1-homogeneous  function with $H_{x}, H_{y} \neq 0$.  Partial differential inequality (\ref{numbers}) holds  on $\mathrm{int}(\mathbb{R}_{+}^{2})$ if and only if one of the following holds: 
 
\begin{align}
& H \quad \text{is a convex function}\, ; \label{p-c}\\
&H(x,y)=Cx^{a}y^{b}, \quad C>0,  \quad b=1-a, \quad a \in (0,1); \label{p-1}\\
&H(x,y)=Cx^{a}y^{-b}, \quad C<0, \quad b=a-1, \quad a \in (1, \infty); \label{p-2}\\
&H(x,y)=Cx^{-a}y^{b}, \quad C<0, \quad b=a+1, \quad a \in (0, \infty). \label{p-3} 
\end{align} 
\end{corollary}
\begin{proof}


Since $H$ is 1-homogeneous we have $H(x,y) = xh(\frac{y}{x})$ for some  $h \in C^{3}(\mathrm{int}(\mathbb{R}_{+}))$. Conditions $H_{x}\neq 0$ and $H_{y} \neq 0$ imply that $h' \neq 0$ and $h(t)-th'(t) \neq 0$. 
 Notice that (\ref{numbers}) takes the following form
  \begin{align}\label{pde1}
  a^{2} \frac{H_{xx}}{H_{x}^{2}}+(1-a^{2}-b^{2})\frac{H_{xy}}{H_{x} H_{y}}+b^{2} \frac{H_{yy}}{H_{y}^{2}} = 
   \frac{(b|h|-|h'|t)^{2}+|h||h'|t(2b+\mathrm{sign}(hh')(a^{2}-b^{2}-1))}{x(h-th')^{2} (h')^{2}}\cdot h''
  \end{align}
where $h=h(t)$ and $t=y/x$. Notice that $(2b \pm (a^{2}-b^{2}-1)) \geq 0$. We have 
\begin{align}\label{diffur1}
(b|h|-|h'|t)^{2}+|h||h'|t(2b+\mathrm{sign}(hh')(a^{2}-b^{2}-1)) \geq 0.
\end{align}
Thus, if $h''(t)\geq 0 $ for all $t>0$ (i.e., condition (\ref{p-c}) holds), or one of the conditions among (\ref{p-1}), (\ref{p-2}) and (\ref{p-3}) hold then clearly the right hand side of  (\ref{pde1}) is nonnegative., i.e., (\ref{numbers}) holds.   Now let us show the converse. 

Assume the right hand side of (\ref{pde1}) is nonnegative. If $h''(t)\geq 0$ for all $t>0$ then $H$ satisfies (\ref{p-c}). Therefore without loss of generality assume that on some interval $I\subset (0, \infty)$ we have $h''<0$.  Thus the right hand side of (\ref{pde1}) is nonnegative on $I$ if and only if the left hand side of (\ref{diffur1}) is zero. This can happen if and only if  $b|h|=|h'|t$ and $|h||h'|t(2b+\mathrm{sign}(hh')(a^{2}-b^{2}-1))=0$. 
Also notice that  if $b|h|=|h'|t$ then $h\neq 0$  on $I$ (since $h'\neq 0$). We consider several cases. 

Suppose $h' t = bh$ on $I$. Then $h=Ct^{b}$ on $I$ for some nonzero $C \in \mathbb{R}$.  Then $\mathrm{sign}(hh')=1$ on $I$, and we obtain that 
$2b+(a^{2}-b^{2}-1)=0$. The last equality implies that either $a+b=1$ or $b-a=1$. Thus we obtain that  $h(t) = Ct^{b}$ on $I$ for some nonzero $C$ with  $a+b=1$ or $b-a=1$.

Suppose $h' t = -bh$ on $I$. Then $h=Ct^{-b}$ on $I$ for some nonzero $C \in \mathbb{R}$. Then $\mathrm{sign}(hh')=-1$ on $I$ and we obtain that 
$2b-(a^{2}-b^{2}-1)=0$. The last equality can happen if and only if $a-b=1$. Thus $h(t)=Ct^{-b}$ on $I$  with some nonzero $C$ and $a-b=1$. 

Thus if $h''<0$ and, thereby, the left hand side of (\ref{diffur1}) is zero on some interval $I$ then the several cases might happen: 1) $a+b=1$ and $h(t)=Ct^{b}$, $C>0$;  2) $b-a=1$ and $h(t)=Ct^{b}$, $C<0$ ; 3) $a-b=1$ and $h=Ct^{-b}$, $C<0$. 
 Notice that non of these strictly concave functions can be glued $C^{2}$ smoothly with a convex function. It follows that $I=(0, \infty)$ (otherwise choose the maximal interval $I$ and consider the value $h''$ at the endpoints of $I$). Thus (\ref{pde1}) is nonnegative if and only if either $h$ is a convex function, or $h$ is a concave function of the form
\begin{align}
&h(t)=Ct^{b}, \quad C>0, \quad a+b=1 \quad \text{and} \quad H(x,y)=Cx^{a}y^{b};\label{ff-1}\\
&h(t) =Ct^{b}, \quad C<0, \quad  b-a=1 \quad \text{and} \quad H(x,y)=Cx^{-a}y^{b};\label{ff-2}\\
&h(t) = Ct^{-b}, \quad C<0, \quad a-b=1 \quad \text{and} \quad H(x,y)=Cx^{a}y^{-b}.\label{ff-3}
\end{align}



\end{proof}

So,  in case of  smooth 1-homogeneous functions there are two instances: $H$ is convex, or $H$ coincides with one of the functions (\ref{p-1}), (\ref{p-2}) and (\ref{p-3}).  Next we describe measures $d\mu$ which satisfy  (\ref{ragaca}) for 1-homogeneous functions $H$. We consider the case when  $H$ is a function of the form (\ref{ff-1}), (\ref{ff-2}) and (\ref{ff-3}).  
\subsubsection{Case of the Pr\'ekopa-Leindler functions}

Functions found in (\ref{ff-1}), (\ref{ff-2}) and (\ref{ff-3}) provide us with the following inequalities.

\begin{corollary}\label{character}
Let $d\mu$ be the Gaussian measure (or the Lebesgue measure). We have 
\begin{align}
\int_{\mathbb{R}^{n}}\esssup_{y \in \mathbb{R}^{n}}\; f^{a}\left(\frac{x-y}{a}\right) g^{1-a}\left( \frac{y}{1-a}\right) d\mu(x) \geq \left( \int f d\mu \right)^{a} \left( \int g d\mu \right)^{1-a}, \quad a \in (0,1) \label{prekopa-1}
\end{align}
for all nonnegative Borel measurable $f,g \in L^{1}(d\mu)$. 
Moreover, if $d\mu$ is even then
\begin{align}
\int_{\mathbb{R}^{n}}\essinf_{y \in \mathbb{R}^{n}}\; f^{a}\left(\frac{x-y}{a}\right) g^{1-a}\left( \frac{y}{1-a}\right) d\mu(x) \leq \left( \int f d\mu \right)^{a} \left( \int g d\mu \right)^{1-a}, \quad a \in (-\infty,0) \cup (1,\infty) \label{prekopa-2}
\end{align}
 for all bounded compactly supported  nonnegative Borel measurable $f,g$ with positive $\int g d\mu$ and  $\int f d\mu $.   
\end{corollary}

\begin{proof}
Inequalities in the corollary follow from the application of Theorem~\ref{mth1} to the functions (\ref{p-1}), (\ref{p-2}) and (\ref{p-3}). The only obstacle to directly apply Theorem~\ref{mth1} is that $H \notin C^{3}(\mathbb{R}^{2}_{+})$.  To avoid this obstacle one needs to consider an auxiliary function $H_{\varepsilon}(x,y)=H(x+\varepsilon,y+\delta)$ for $\varepsilon, \delta>0$ and then send $\varepsilon,  \delta \to 0$ (see the similar discussions in (\ref{auxiliary})).  

Case of the Lebesgue measure follow from the Gaussain measure and the fact that $H$ is 1-homogeneous. Indeed, we can test inequalities in the corollary for the  following test functions $f_{\lambda}(x)=f(\lambda x)$ and $g_{\lambda}(x)=g(\lambda x)$. By making  change of variables and using 1-homogeneity of $H$  we can  send $\lambda \to \infty$ and obtain the desired result. 
\end{proof}

 Inequality (\ref{prekopa-1}) is the classical Pr\'ekopa--Leindler inequality \cite{prekopa-and-leindler1, prekopa-and-leindler2}.  Among its many applications we should mention a remarkable paper \cite{bo-3}.
Stability of (\ref{prekopa-1}) was studied in \cite{bkb}.  
  The inequality implies that the marginals of log-concave measures are log-concave. For a local version of the latter fact we refer the reader to  \cite{bbn}  (see also \cite{dce} for the complex setting).  An extension of the inequality was obtained in \cite{dce-bm}.

 Inequality (\ref{prekopa-2}) can be understood as an extension of the classical Pr\'ekopa--Leindler inequality for $a \notin [0,1]$. In fact,   one can show that (\ref{prekopa-1}) and (\ref{prekopa-2})  are equivalent  if instead of essential infimum in (\ref{prekopa-2}) we would have only infimum.

  It is the remarkable result of Borell \cite{Borell2} that (\ref{prekopa-1}) holds if and only if $d\mu$ has a density woth respect to the Lebesgue measure on some affine hyperplane, and this density is logarithmically concave function.  One can show that the weaker version of (\ref{prekopa-2}), i.e., when essential infimum is replaced by infimum, also holds for even log-concave measures. 
  
  Finally we would like to mention that even though among $C^{3}$ smooth 1-homogenous functions $H$ with nonvanishing $H_{x}$ and $H_{y}$ there are only two instances either $H$ is convex or $H$ is of the form (\ref{p-1}), (\ref{p-2}) and (\ref{p-3}),  it is not the case in general if we drop the assumption of smoothness. We can always take  the maximum of any two functions which satisfy (\ref{cl-max}). Indeed, next proposition says that (\ref{ragaca}) is closed under taking maximum. 
\begin{proposition}\label{cl-max}
If $H_{1}$ and $H_{2}$ satisfy (\ref{ragaca})  then $H=\max\{H_{1},H_{2}\}$ also satisfies (\ref{ragaca}). 
\end{proposition}
\begin{proof}
Indeed, suppose $H\left( \int f, \int g \right) =  H_{1}\left( \int f, \int g \right)$ then since $H \geq H_{1}$ we have 
\begin{align*}
\int \esssup_{y}\; H \left( f \left( \frac{x-y}{a}\right), g\left(\frac{y}{b}\right) \right) d\mu(x) \geq \int \esssup_{y}\; H_{1} \left( f \left( \frac{x-y}{a}\right), g\left(\frac{y}{b}\right) \right) d\mu(x) \geq H_{1}\left( \int f d\mu, \int g d\mu \right).
\end{align*}
\end{proof}

 \section{Appendix}

 
 
 The following remarkable result belongs to Borell \cite{Borell2}. 
  \begin{theoremB}\label{borel-th}
  Let $\varphi : \mathbb{R}^{n} \times \mathbb{R}^{n}  \to \mathbb{R}^{n}$ be a continuously differentiable function such that 
  \begin{align*}
 &\varphi = (\varphi^{1}, \ldots, \varphi^{n});\\
 &\varphi^{k}(x_{1}, x_{2})=\varphi^{k}(x_{1}^{k}, x_{2}^{k}),\quad  x_{i}=(x_{i}^{1},\ldots, x_{i}^{n})\quad \text{for} \quad  i=1,2; \quad  k=1,\ldots, n; \\
 &\frac{\partial \varphi^{k}}{\partial x_{i}^{k}} >0, \quad i=1,2, \quad k=1,\ldots, n. 
  \end{align*}
 Let $f, g, h \geq 0$ and $f,g,h  \in L^{1}_{loc}(\mathbb{R}^{n})$. Further  suppose $\Phi : [0,\infty)\times [0,\infty)\to [0,\infty]$ is a continuous 1-homogeneous function and increasing in each variable. Then the inequality 
 \begin{align}\label{bo-je}
 \int^{*}_{\mathbb{R}^{n}} h \mathbbm{1}_{\varphi(A,B)} dm \geq \Phi \left( \int_{\mathbb{R}^{n}}^{*} f \mathbbm{1}_{A} dm, \int_{\mathbb{R}^{n}}^{*} g \mathbbm{1}_{B} dm  \right)
 \end{align}
 holds for all nonempty $A, B \subset \mathbb{R}^{n}$ if and only if there are sets $\Omega_{1}, \Omega_{2} \subset \mathbb{R}^{n},$ $m(\mathbb{R}^{n}\setminus \Omega_{1})= m(\mathbb{R}^{n}\setminus \Omega_{2} )=0$ such that 
  \begin{align}\label{borell-integral}
  h(\varphi(x,y)) \prod_{k=1}^{n} \left(\frac{\partial \varphi_{k}}{\partial x_{k}} \rho_{k} + \frac{\partial \varphi_{k}}{\partial y_{k}}\eta_{k} \right) \geq \Phi\left(f(x) \prod_{k=1}^{n}\rho_{k},\;  g(y) \prod_{k=1}^{n}\eta_{k} \right)
  \end{align}
  for all $x \in \Omega_{1}$,  $y \in \Omega_{2}$, $\rho_{1},\ldots,\rho_{n}>0$ and for every $\eta_{1},\ldots, \eta_{n}>0$.  Moreover if (\ref{bo-je}) is holds then  $\Omega_{1} = \mathrm{supp}\, f$ and $\Omega_{2} = \mathrm{supp}\, g$ will do. 
   \end{theoremB}
   
    Borell obtained the theorem in more general case when one can include arbitrary number of  test functions and $\varphi$ can be  defined only on some subdomains of $\mathbb{R}^{n}$.

  Let us consider a particular case when $n=1$. Since in the current  paper we are interested  when the inequalities of the form $h(ax+by) \geq H(f(x),g(y))$ imply its integral version $\int h \geq H(\int f, \int g)$ then in order to apply Borell's result we should take $\varphi(x,y)=ax+by$ for $x,y \geq 0$. Then (\ref{borell-integral}) takes the following form $h(ax+by) (a\rho+b \eta) \geq \Phi(f(x)\rho, g(y) \eta)$. 
  The form (\ref{bo-je}) reduces to the form (\ref{ragaca}) if  $h(t) = \sup_{ax+by=t} \Phi(f(x)\mathbbm{1}_{A},g(y)\mathbbm{1}_{B})$ and $h$ is supported on $\mathbbm{1}_{aA+bB}$. This may happen if and only if $\Phi(0,0)=\Phi(0,1)=\Phi(1,0)=0$. Since $A$ and $B$ is arbitrary we obtain  that  $h(ax+by)=\Phi(f(x),g(y))$. Therefore the last condition takes the form 
  \begin{align}\label{borell-integral2}
  \Phi(f(x),g(y)) (a\rho+b\eta) \geq \Phi(f(x)\rho, g(y) \eta).
  \end{align}
  Thus if (\ref{borell-integral2}) holds for all $\eta, \rho >0$ and all nonnegative $f(x), g(y)$ then we obtain  the integral inequality 
  \begin{align}\label{integral-version}
  \int_{\mathbb{R}}^{*} \sup_{ax+by=t} \Phi(f(x),g(y)) dt \geq \Phi\left(\int_{\mathbb{R}} f dx, \int_{\mathbb{R}} g dx \right) \quad \text{for all nonnegative} \quad f,g \in L^{1}.
  \end{align}
 Since $\sup_{ax+by=t} \Phi(f(x),g(y))$ may not be measurable we should understand the integral in the left hand side of (\ref{integral-version}) as an upper integral. 
 \begin{proposition}
 Let $\Phi \in C^{1}(\mathrm{int}(\mathbb{R}^{2}_{+}))\cap C(\mathbb{R}^{2}_{+})$ be 1-homogeneous,  nonnegative and   increasing in each variable. Assume $\Phi(0,0)=\Phi(0,1)=\Phi(1,0)=0$. If $\Phi$  satisfies (\ref{borell-integral2}) for all positive $\rho, \eta, f(x), g(y)$, and with some positive $a,b$ such that $a+b=1$ then $\Phi(x,y)=Cx^{a}y^{b}$.
 \end{proposition}
  
  \begin{proof}
$\Phi$ is one homogeneous therefore  $\Phi(p,q)=p\, m\left( \frac{q}{p}\right)$ for some nonnegative increasing function $m \in C^{1}(0,\infty)$.  (\ref{borell-integral2}) simplifies to $m(s)(a+bt) \geq m(st)$ for all $s,t > 0$. Let $st=u$ and $s=v$ then we obtain $m(v)(av+bu) \geq m(u) v$. Set $ u  = v+ \varepsilon$. Then by Taylor's expansion we obtain that for sufficiently small $\varepsilon$ we have  $m(v) v + m(v) b \varepsilon \geq v m(v) + v m'(v)\varepsilon + o(\varepsilon)$. Since $\varepsilon$ can be negative as well we obtain $b m(v) = v m'(v)$ and hence $m(v) = Cv^{b}$ for some $C>0$. Therefore $\Phi(p,q)=Cp^{a}q^{b}$. 
  \end{proof}
  
  Thus the corollary shows that  in the particular case $\varphi = ax+by$ the functions which satisfy the assumption of Borell's theorem (\ref{borell-integral2}) and hence would give us integral inequality (\ref{integral-version}) are of the form $\Phi(x,y) = x^{a} y^{b}$. The reader can recognize that this is the instance of the  Pr\'ekopa--Leindler inequality.  Notice that this also confirms our result: in Subsection~\ref{lebeg-hom} we have found that $\Phi$ has to be convex function or $\Phi$ has to be function of the form (\ref{p-1}), (\ref{p-2}) and (\ref{p-3}). Since  in application of Borell's theorem we require that  $\Phi(0,0)=\Phi(0,1)=\Phi(1,0)$, and $\Phi \geq 0$ then the only possibility is $\Phi(x,y)=x^{a}y^{b}$. Indeed,  1-homogeneous convex nonnegative function $\Phi(x,y)$ on $\mathbb{R}^{2}_{+}$  with values zero at the points $(0,0)$, $(0,1)$ and $(1,0)$ must be identically zero. 
 
 \subsection{Sketch of the proof of Theorem~A}\label{atrakebs}
Without loss of generality we can assume that $C$ is identity matrix. Indeed, we can  denote $\tilde{A}_{i} := C^{1/2}A_{i}$ for $i=1,2,3$, and $\tilde{A} := (\tilde{A}_{1}, \tilde{A}_{2}, \tilde{A}_{3})$, and make change of variables $\tilde{x}:=xC^{-1/2}$ in the left hand side of (\ref{ivo-2}). Thus, it is enough to show that  $A^{*} A \bullet \mathrm{Hess}\, B \leq  0$ if and only if 
\begin{align}
&\int_{\mathbb{R}^{k}} B(u_{1}(xA_{1}), u_{2}(xA_{2}), u_{3}(xA_{3}))d\gamma_{k}(x)\geq \label{mokla}\\
&B\left(\int_{\mathbb{R}^{k_{1}}}u_{1}(y\sqrt{A_{1}^{*}A_{1}}) d\gamma_{k_{1}}(y),\int_{\mathbb{R}^{k_{2}}}u_{2}(y\sqrt{A_{2}^{*}A_{2}})d\gamma_{k_{2}}(y), \int_{\mathbb{R}^{k_{3}}}u_{3}(y\sqrt{A_{3}^{*}A_{3}})d\gamma_{k_{3}}(y)\right).\nonumber
\end{align}
Next, denote $\tilde{u}_{j}(x) := u_{j}(xA_{j})$, and let $P_{t}\tilde{u}_{j}$ be its heat extension, i.e., $\partial_{t} P_{t}\tilde{u}_{j} = \Delta P_{t}\tilde{u}_{j}$, and $P_{0}\tilde{u}_{j}=\tilde{u}_{j}$.  We will need the following key identities
\begin{align*}
&P_{t}\tilde{u}_{j}(x) =\int_{\mathbb{R}^{k}} \tilde{u}_{j}(x+y\sqrt{2t})d\gamma_{k}(y)\stackrel{(*)}{=} \int_{\mathbb{R}^{k_{j}}} u_{j}(xA_{j} + \tilde{y} (2t A_{j}^{*}A_{j})^{1/2}) d\gamma_{k_{j}}(\tilde{y}),\\
&\nabla P_{t} \tilde{u}_{j}(x) = \left(\int_{\mathbb{R}^{k_{j}}}\nabla u_{j}(xA_{j}+\tilde{y}(2tA^{*}_{j}A_{j})^{1/2})d\gamma_{k_{j}}(\tilde{y})\right)A_{j}^{*}=: (P^{j}_{t}\nabla u_{j}(x)) A_{j}^{*},\quad j=1,2,3, 
\end{align*}
where $\nabla u_{j}(z):= \nabla u_{j}(y) | _{y =z}$. 
Equality $(*)$ follows from a property of the Gaussian measure, namely, 
$$
\int_{\mathbb{R}^{k}}u_{j}(yA_{j}) d\gamma_{k}(y) = \int_{\mathbb{R}^{k_{j}}}u(\tilde{y} \sqrt{A_{j}^{*}A_{j}}) d\gamma_{k_{j}}(\tilde{y}) \quad \text{for} \quad j=1,2,3. 
$$
 
Next, let $\vec{u}(x):=(\tilde{u}_{1}(x), \tilde{u}_{2}(x), \tilde{u}_{3}(x))$ and $P_{t} \vec{u}(x) =(P_{t} \tilde{u}_{1}(x), P_{t} \tilde{u}_{2}(x), P_{t} \tilde{u}_{3}(x))$. If we test inequality (\ref{mokla}) on the functions  $f_{j}(y):=u_{j}(xA_{j}+y\sqrt{2t})$, we obtain that (\ref{mokla}) is equivalent to the following inequality 
\begin{align*}
V(x,t):=B(P_{t}\vec{u}(x))  -P_{t}B(\vec{u}(x)) \geq  0
\end{align*}
for all $x \in \mathbb{R}^{k}$ and all $t \geq 0$, and  the case $x=0, t=1/2$ gives exactly (\ref{mokla}).  Denote 
\begin{align*}
P_{t}\nabla  \vec{u}(x) := (P_{t}^{1}\nabla u_{1}(x), P_{t}^{2}\nabla u_{2}(x), P_{t}^{3}\nabla u_{3}(x))
\end{align*}
It follows from the straightforward calculation that 
\begin{align*}
(\Delta - \partial_{t})V(x,t)= (P_{t} \nabla \vec{u})\, (A^{*}A \bullet \mathrm{Hess}\, B(P_{t} \vec{u}))  (P_{t} \nabla \vec{u})^{*}. 
\end{align*}
Therefore, if $A^{*} A \bullet \mathrm{Hess}\; B \leq 0$ then $(\Delta-\partial_{t})V\leq 0$. Since $V(x,0)=0$, it  follows from the classical maximum principle $V(x,t)\geq 0$ for all $x \in \mathbb{R}^{k}$, $t\geq 0$.  

On the other hand if (\ref{mokla}) holds, then we have explained that   $V(x,t)\leq 0$ for all $x \in \mathbb{R}^{k}$ and $t\geq 0$. Therefore
\begin{align*}
0\leq \lim_{t \to 0+} \frac{V(x,t)}{t} =  \lim_{t \to 0+} \frac{V(x,t)-V(x,0)}{t}=- (P_{0} \nabla \vec{u}(x))\, (A^{*}A \bullet \mathrm{Hess}\, B(P_{0} \vec{u}(x)))  (P_{0} \nabla \vec{u}(x))^{*}.
\end{align*}
 Since $P_{0}\vec{u}(x)=\vec{u}(x)$, $P_{0}\nabla \vec{u} = (\nabla u_{1}(xA_{1}), \nabla u_{2}(xA_{2}), \nabla u_{3}(xA_{3}))$, and $\vec{u}$ is arbitrary, we obtain $A^{*}A \bullet \mathrm{Hess}\, B \leq 0$.

\end{document}